\documentclass[11pt]{preprint}
\usepackage[full]{textcomp}
\usepackage[osf]{newtxtext} 
\usepackage[cal=boondoxo]{mathalfa}

\usepackage{colortbl}
\usepackage{tikz-cd} 
\usepackage{float}
\usetikzlibrary{cd} 
\usepackage[all,cmtip]{xy}
\usepackage{amssymb}
\usepackage{mathtools}
\usepackage{hyperref}
\usepackage{breakurl}
\usepackage{mhenvs}
\usepackage{mhequ} 
\usepackage{mhsymb}
\usepackage{booktabs}
\usepackage{tikz}
\usepackage{tcolorbox}
\usepackage{mathrsfs}
\usepackage[utf8]{inputenc}
\usepackage{longtable}
\usepackage{wrapfig}
\usepackage{rotating} 
\usepackage{subcaption}
\usepackage{mathrsfs}
\usepackage{epsfig}
\usepackage{microtype}
\usepackage{comment}
\usepackage{wasysym}
\usepackage{centernot}
\usepackage{enumitem}
\usepackage{bm}
\usepackage{stackrel}
\usepackage{amsmath}

\usepackage{graphicx}
\usepackage{axodraw}
\usepackage{xspace}
\usepackage{epsfig} 
\usepackage{axodraw} 
\usepackage{xspace} 
\usepackage[toc,page]{appendix} 
\usepackage[all,cmtip]{xy} 
\usepackage{relsize} 
\usepackage{shuffle} 

\DeclareMathAlphabet{\mathbbm}{U}{bbm}{m}{n}

\DeclareFontFamily{U}{BOONDOX-calo}{\skewchar\font=45 }
\DeclareFontShape{U}{BOONDOX-calo}{m}{n}{
  <-> s*[1.05] BOONDOX-r-calo}{}
\DeclareFontShape{U}{BOONDOX-calo}{b}{n}{
  <-> s*[1.05] BOONDOX-b-calo}{}
\DeclareMathAlphabet{\mcb}{U}{BOONDOX-calo}{m}{n}
\SetMathAlphabet{\mcb}{bold}{U}{BOONDOX-calo}{b}{n}

\setlist{noitemsep,topsep=4pt,leftmargin=1.5em}

\DeclareMathAlphabet{\mathbbm}{U}{bbm}{m}{n}

\DeclareMathAlphabet{\mcb}{U}{BOONDOX-calo}{m}{n}
\SetMathAlphabet{\mcb}{bold}{U}{BOONDOX-calo}{b}{n}
\DeclareFontFamily{U}{mathx}{\hyphenchar\font45}
\DeclareFontShape{U}{mathx}{m}{n}{
      <5> <6> <7> <8> <9> <10>
      <10.95> <12> <14.4> <17.28> <20.74> <24.88>
      mathx10
      }{}
\DeclareSymbolFont{mathx}{U}{mathx}{m}{n}
\DeclareMathSymbol{\bigtimes}{1}{mathx}{"91}

\def\emptyset{{\centernot\ocircle}}

\def\id{\mathrm{id}}

\def\CH{\mathcal{H}}

\def\CC{\mathcal{C}}

\newcommand{\mbX}{\mathbf{X}}

\newcommand{\mbZ}{\mathbf{Z}}

\newcommand{\abs}[1]{\left\lvert #1\right\rvert}

\theorembodyfont{\rmfamily}

\def\mail#1{\burlalt{#1}{mailto:#1}}

\begin{document}

\title{Flows driven by multi-indices Rough Paths}
\date{}

\author{Carlo Bellingeri$^1$, Yvain Bruned$^2$, Yingtong Hou$^2$}
\institute{Université de Haute-Alsace, IRIMAS, 68200 Mulhouse, France\and Université de Lorraine, IECL, F-54000 Nancy, France
  \\
Email:\ \begin{minipage}[t]{\linewidth}
	\mail{carlo.bellingeri@uha.fr},\\
\mail{yvain.bruned@univ-lorraine.fr},
\\\mail{yingtong.hou@univ-lorraine.fr}
\end{minipage}}

\maketitle

\begin{abstract}
\ \ \ \ In this work, we introduce a solution theory for scalar-valued rough differential equations driven by multi-indices rough paths.  To achieve this task, we will show how the flow approach using the log-ODE method introduced by Bailleul fits perfectly in this setting. In addition, we  also describe   the action of the translation of  multi-indices rough paths at the level of rough differential equations.
\end{abstract}

\setcounter{tocdepth}{2}
\tableofcontents

\setlength{\parskip}{0.5\baselineskip}
\section{ Introduction }
We consider the scalar-valued rough differential equation (RDE) 
\begin{equs}\label{RDE}
\mathrm{d}Y_t=f_0(Y_t)\mathrm{d}t+f(Y_t)dX_t=f_0(Y_t)\mathrm{d}t+ \sum_{i=1}^d f_i(Y_t)\mathrm{d}X^i_t\,, \quad y_0 \in \mathbb{R},
\end{equs}
where  $f\colon  \mathbb{R}\to \mathbb{R}^{d}$, $f=(f_1\,, \ldots\,, f_d)$ and $f_0: \mathbb{R} \rightarrow \mathbb{R}$ are sufficiently smooth non-linearities, and $X^i\colon[0,T]\to \mathbb{R}$ is a source term of H\"older regularity $ \gamma \in (0,1) $ for $i=1,\ldots,d$. Solution theory to these equations has been given in the context of rough paths introduced in \cite{lyons1998,Gubinelli2004,Gub06} mainly using geometric and non-geometric rough paths which at the combinatorial level corresponds to describe iterated integrals with words or decorated trees. See \cite{FrizHai} for a gentle introduction to rough paths. A more general framework for rough paths has been given in \cite{Tapia20} where one can use a Hopf algebra that provides the natural properties for iterated integrals. But there are few examples beyond the shuffle Hopf algebra for words \cite{lyons1998,Gubinelli2004} and the Butcher-Connes-Kreimer Hopf algebra for decorated trees \cite{Butcher72,CK1,Gub06,HK15}: One can mention Quasi-geometric rough paths in \cite{Bel23} where stochastic brackets are described via contracted letters and \cite{Manchon20} for planar decorated trees.

Recently, a new combinatorial set called multi-indices has been proposed in the context of singular stochastic partial differential equations. It appears in \cite{OSSW} for the study of quasi-linear equations via the theory of regularity structures which has been originally based on decorated trees in \cite{reg,BHZ}. This new structure proposes an optimal description of the expansion of solutions for scalar-valued equations. Many works around this structure provide an interesting construction of the renormalised models \cite{LOT,LOTT,BL24}. This structure also appears naturally in numerical analysis \cite{BHE24} and in quantum field theory \cite{BH25}.
So far, there is no solution theory of RDEs driven by multi-indices rough paths since multi-indices represent sums of iterated integrals, making it not obvious to establish a fixed point via Picard iteration.

The main aim of this paper is to provide a solution theory when the equation \eqref{RDE} is interpreted via multi-indices rough paths introduced in \cite{Li23}. Our strategy is to use the approximated flow approach provided by Ismaël Bailleul in \cite{B15,B15b,B19,B21}. This approach is very successful in recovering the solution theory for rough differential equations driven by geometric and non-geometric rough paths. It is comparable in spirit to the Davie notion of solution \cite{Davie} and the $\log$-$\mathrm{ODE}$ method \cite{CG96}. It has also been used in a very abstract context where one has to check only two algebraic properties on the elementary differentials to get the solution in \cite{KL23}. We follow this approach for getting the solution theory with multi-indices rough paths. Before stating our main result, let us briefly describe the idea of the flow method. By performing a formal Taylor expansion of  $Y$, the solution of \eqref{RDE}, around a point $Y_s=y$ of the dynamics, the solution should be well  approximated by the quantity
\begin{equs} \label{Butcher_expansion}
\sum_{\substack{ z^{\beta}\in \mathfrak{M}\\0\leq \vert z^{\beta}\vert_{\gamma}\leq N_{\gamma}}}\frac{\Upsilon_f[z^{\beta}](y)}{S(z^{\beta})} \mbX_{s,t}( z^{\beta}) 
\end{equs}
where $ \mathfrak{M} $ is a suitable set of multi-indices and $N_{\gamma}=\lfloor \gamma^{-1}\rfloor$.  Here $ S(z^{\beta}) $ is a symmetry factor, $\Upsilon_f[z^{\beta}]$ is an elementary differential associated with $z^{\beta}$, and $ \mbX_{s,t} $ is a rough path above $X$, a collection of iterated polynomials  built from $X$.The size of these combinatorial objects $ |\cdot|_{\gamma} $ is bounded by $N_{\gamma}$ which is connected to the regularity $\gamma$ of the path $X$ when the trajectories of $X$ are for instance $C^1$.  The expansion \eqref{Butcher_expansion} is written with the B-series formalism \cite{Butcher72} which is quite efficient for describing numerical methods for ODEs. 
Then, we take the logarithm of our rough path $ \mbX_{s,t} $ 
\begin{equs}	
		\boldsymbol\Lambda^{\mbX}_{s,t}= \log_{\star_{{N}_\gamma}}(\mbX_{s,t})
	\end{equs}
where $\star_{{N}_\gamma}$ refers to the truncation of order $N_{\gamma}$ of a suitable product $\star$.
This allows to write an equation that will be almost  satisfied by \eqref{Butcher_expansion}: 
\begin{equation}\label{def_logode_1}
	\left\{\begin{aligned}
		&\dot{Z}_r=  \sum_{\substack{ z^{\beta}\in \mathfrak{M}\\1\leq \vert z^{\beta}\vert_{\gamma}\leq N_{\gamma}}}\frac{\Upsilon_f[z^{\beta}](Z_r)}{S(z^{\beta})} \Lambda^{\mbX}_{s,t}(z^{\beta})\\&Z_0=y\,
	\end{aligned}\right.
\end{equation}
where the map  $\mu^{\mbX}_{s,t}(y)=Z_1$ is the log-ODE almost flow. Then one can get local well-posedness of the equation \eqref{RDE}, if the flow map is an approximated flow in the sense that for $ y \in\mathbb{R} $
\begin{equs} \label{approximated flow}
	\abs{\mu^{\mbX}_{s,t}(y)-\mu^{\mbX}_{r,t}\circ\mu^{\mbX}_{s,r} (y)}\lesssim \abs{t-s}^{a}
\end{equs}
with $a > 1$ and $\lesssim$ stands for inequality up to a constant. The property \eqref{approximated flow} is based on two key identities on the elementary differentials:
\begin{equs} \label{key_identities}
	\begin{aligned}
		\Upsilon_f[u \star v](\phi) 
	& = 
	\Upsilon_f[u]  \circ \Upsilon_f[v](\phi)
\\		\Upsilon_f[ \Delta_\shuffle u] (\phi \otimes \psi) & = \Upsilon_f[u] (\phi \psi)\,,
\end{aligned}
\end{equs}
where  $u,v$ corresponds to linear span of forests of populated multi-indices,   $ \Upsilon_f[u](\varphi) $ is the extension of the elementary differentials to the ones composed with a smooth enough function $\varphi$, and $ \Delta_\shuffle $ is the deshuffle coproduct. One can find more precise definitions in Sections \ref{sec::2.1} and \ref{sec::2.3}. The finding that these two identities are the key to the existence of the ``almost flow"  has been discovered in a general context in \cite{KL23}. Let us mention that they are very close in spirit to the Newtonian maps introduced in \cite{Lejay}.

Now, we present the main statement of the paper. Firstly, we need to introduce some notations describing the functional spaces to set up the theory.
For any $\alpha>1$ with integer part $\lfloor\alpha\rfloor$ and fractional part $\{\alpha\}$,  we say that a function $f\colon  \mathbb{R}\to \mathbb{R}^d $ is a $\alpha$-H\"older vector field, with notation $f\in \CC^{\alpha}(\mathbb{R}, \mathbb{R}^d  )$,  if $f$ is of class $C^{\lfloor\alpha\rfloor}$, that is $\lfloor\alpha\rfloor$ times continuously differentiable, with an $\{\alpha\}$-H\"older continuous  derivative of order $\lfloor\alpha\rfloor$. We introduce the Banach space of bounded $\alpha$-H\"older vector fields $\CC^{\alpha}_b(\mathbb{R}, \mathbb{R}^d )$  such that
\begin{equation}
\label{DefGammaNorm}
\|f\|_{\alpha} = \sum_{r=0}^{\lfloor\alpha\rfloor} \sup_{x}\big|f^{(r)}(x)\big| +\sup_{x\neq z}\frac{\Vert f^{(\lfloor\alpha\rfloor)}(x)- f^{(\lfloor\alpha\rfloor)}(z) \Vert}{|x-z|^{\{\alpha\}}}<\infty\,,
\end{equation}
where   $\Vert\cdot\Vert$ is the euclidean norm on $\mathbb{R}^d$. In addition to the space of vector fields We will use the shorthand notation $\CC^{\alpha}$ and  $\CC^{\alpha}_b$ to denote $\CC^{\alpha}(\mathbb{R}, \mathbb{R}  )$ and $\CC^{\alpha}_b(\mathbb{R}, \mathbb{R})$. As in \cite{BCF,BCFP,Bel22}, following \cite{Li23}, for any  linear map $ \ell_i : \mathfrak{M} \rightarrow \mathbb{R} $ with $ i \in \lbrace 0,...,d \rbrace $, one can define a translation map $T_{\ell}$ whose action on $f$ is given  by
\begin{equs}\label{eq_translated_field}
	 f_i^{\ell} = \sum_{z^{\beta} \in \mathfrak{M}} \frac{\ell_i(z^{\beta})}{S(z^{\beta})} \Upsilon_f[z^{\beta}]\,.
\end{equs}
The following main theorem is split in the paper into Theorem \ref{thm_well-posedness} and Theorem \ref{thm:renormalization_RDE}.

\begin{theorem}
Given $y_0\in \mathbb{R}$  and $\mathbf{X}$ a multi-index rough path of regularity $\gamma\in(0,1) $, for any $f\in \mathcal{C}^{\alpha_1}(\mathbb{R}, \mathbb{R}^d  )$  and $f_0 \in \mathcal{C}^{\alpha_2}$ with $\alpha_1>$ $N_{\gamma}$ (or $N N_{\gamma}$ when one considers the translated RDEs introduced below) and $ \alpha_2>1$ the rough differential equation \eqref{RDE} has a unique local solution flow. In case $f\in \CC^{\alpha_1}_b(\mathbb{R}, \mathbb{R}^d )  $ or it is linear and $f_0 \in \mathcal{C}^{\alpha_2}_b$ or it is linear, the solution is global. Moreover, for any  family of linear maps $ \ell_i : \mathfrak{M} \rightarrow \mathbb{R} $ with $ i \in \lbrace 0,...,d \rbrace $ and such that $\ell_i(z^\alpha) = 0$ for any $|z^\alpha|_{\gamma} > N$,  if we translate the rough path $\mathbf{X}$ as $T_{\ell}\mathbf{X}$ then the rough differential equation driven by $T_{\ell}\mathbf{X}$ is equivalent to the same   rough differential equation with $ f $ replaced by $f^{\ell}$ from \eqref{eq_translated_field}.
\end{theorem}

\begin{remark}
It turns out that the  translation $T_{\ell}$ maps satisfy the following identity:
\begin{equs} \label{general_translation}
    T_{\ell} (u \star v) = (T_{\ell} u) \star (T_{\ell} v)
    \end{equs}
    This property could be added to the identities \eref{key_identities}, taken from  of \cite{KL23} for building general translated  almost flows in a Hopf algebraic context.
    \end{remark}
    
Let us outline the paper by summarising the content of its sections. In Section~\ref{Sec::2},
we begin with some algebraic preliminaries and recall the basic definitions of multi-indices in the context of differential equations. We first give the general form of multi-indices \eqref{multi-indices} and then introduce their symmetry factor \eqref{symmetry}, inner product \eqref{eq:pairing} and the product $\star_2$ \eqref{definition_star_2} that we call simultaneous grafting.
These definitions allow to set the key product $\star$ in Definition \ref{star_product}. In the second part of this section, we describe multi-indices rough paths in Definition \ref{dfn:genRP} where a truncated version of the product $\star$ is used for giving Chen relation in \eqref{eq:chen}. We define its logarithm in \eqref{logarithm_rough} and after introducing a suitable norm in \eqref{defn_Norm}, we exhibit estimates on the norm of the logarithm of a multi-indices rough path in Proposition \ref{prop_estimate}. The last part of the section is dedicated to some algebraic identities on the vector fields that will be used crucially in the solution theory. We extend the elementary differentials on multi-indices \eqref{elementary_differential_ODE} to elementary differential composed with $\phi \in \CC^\infty(\mathbb{R},\mathbb{R})$ in \eqref{identification_1} and \eqref{identification_2}. The first key identity is the morphism property of the elementary differentials with the product $\star$ in Proposition \ref{morphism_element}. The second one is a Leibniz type identity given in Proposition \ref{Leibniz_identity}.

In Section \ref{Sec::3}, we develop the solution theory for multi-indices rough differential equations via rough flows. We first define an approximated expansion of the solution in \eqref{approximate_exp2}. From this expansion, we define a local flow solution in Definition \ref{DefnGeneralRDESolution}. Then, with the logarithm of a rough path in Definition \ref{def_loga_flow}, we consider  the log-ODE almost flow. These two objects, the approximated expansion and the almost flow, are connected in Proposition~\ref{main_prop}. In Proposition \ref{properties_lamost_flow}, we prove key properties for the almost flow in the rough path context. They rely on the algebraic identities in Propositions \ref{morphism_element} and \ref{Leibniz_identity}. Then, we are able to prove one of the main theorems of this paper, Theorem \ref{thm_well-posedness}, establishing well-posedness for rough differential equations.
We finish the section by recalling the Davie notion of solutions in Definition \ref{Davie_solution} and prove that this notion of solution coincides with the notion introduced with the almost flow in Proposition \ref{prop_Davie}.

In Section \ref{Sec::4}, 
we start by defining the translation map in \eqref{translation_map} that deforms a rough path and the regularity of the translated rough path can be measured from Proposition \ref{prop:regularity_translation}. Since the translation can be described by a Hopf algebra, we will recall the insertion product for multi-indices introduced in \cite{Li23} and then its dual coproduct $\Delta^{\!-}$ \eqref{eq:adjoint}. In particular, the adjoint relation between the simultaneous  insertion $\star_1$ and $\Delta^{\!-}$ is at the core of the proof of the Theorem \ref{thm:renormalization_RDE}, which derives the rough differential equations obtained from a translation map.

%


\subsection*{Acknowledgements}

{\small
	C.B., Y.B. and Y. H. gratefully acknowledge funding support from the European Research Council (ERC) through the ERC Starting Grant Low Regularity Dynamics via Decorated Trees (LoRDeT), grant agreement No.\ 101075208. Views and opinions expressed are however those of the author(s) only and do not necessarily reflect those of the European Union or the European Research Council. Neither the European Union nor the granting authority can be held responsible for them.  Y.B. also thanks the ``Institut des Hautes Études Scientifiques" (IHES) for a long research stay from 7th of January to 21st of March 2025, where part of this work was written.
} 

\section{Rough paths and elementary differentials over multi-indices}
\label{Sec::2}
We start our analysis by recalling the algebraic and analytic structures of rough paths based on multi-indices, as introduced in  \cite{Li23}. To keep the notion  of rough path over multi-indices more consistent with the general notion of rough path over a hopf algebra, see e.g. \cite[Def. 2.1.]{Bel23},
we will describe the multi-indices Hopf algebra using the recent approach from \cite{BHE24,BH24}.

\subsection{Algebraic preliminaries}
\label{sec::2.1}
Let us rewrite the rough differential equation \eqref{RDE} in the integral form
\begin{equs}\label{eq:integral}
	Y_t= Y_0 + \int_0^t f_0(Y_r)\mathrm{d}r + \int_0^t f(Y_r)\mathrm{d}X_r.
\end{equs}
Assume that $f$ is analytical. The Taylor expansion of the vector field around $Y_0$ reads
\begin{equs} \label{Taylor}
	f_i(u) = \sum_{k \geq 0} \alpha_{i,k}(Y_0)(u-Y_0)^k,
\end{equs}
where $\alpha_{i,k}$ represents the $k$-th order derivative of $f_i$ divided by $k!$ for $i=0,\ldots,d$.
Then one can expand the solution formally by plugging back the \eqref{Taylor} to \eqref{eq:integral} and expanding $Y_r$ recursively.
It can be observed that, for a scalar-valued vector field $f$, due to the iteration described above, all the Taylor coefficients in the expansion of the solution collapse to monomials of ordinary derivatives of $f$ with respect to $Y$. This allows us to use the combinatorial object multi-index which was initially introduced in the context of scalar singular SPDEs \cite{OSSW, LOT,LO23,LOTT} to represent the expansion.

The multi-index considered in the paper has the following form
	\begin{equs}\label{multi-indices}
		z^\beta = \prod_{(i ,k) \in  [0,d] \times \mathbb{N}} 
		z_{(i,k)}^{\beta(i,k)}
	\end{equs}
	which is a monomial of the abstract variable $z_{(i,k)}$ where $k$ is the arity of the abstract variable, and 
	the index $i$ representing the ``type" of the abstract variable takes value in $[0,d] := \{0,1,\ldots,d\}$ as we have $(d+1)$-dimensional driven process $X$ in which $X^0 = t$. The multi-index $\beta: [0,d] \times \mathbb{N} \rightarrow \mathbb{N}$ has the entry $\beta(i,k)$ evaluating the frequency of the variable $z_{(i,k)}$ in the multi-index $z^\beta$. 
	Moreover, we assume that $\beta$ has finitely many non-zero entries (finitely supported).
	For conciseness in notations, we make the identification between the space of the multi-index $\beta$ and the space of multi-index $z^\beta$ through \eqref{multi-indices}.
 Here the product works as the usual monomial multiplication
\begin{equs}
	z^\alpha z^\beta = z^\beta z^\alpha = z^{\alpha + \beta}  \quad \text{and} \quad \prod_i z^{\beta_i} = z^{\sum_i \beta_i}
\end{equs}
and we call it a ``multi-index product".

On the other hand, we can construct the symmetric algebra of multi-indices in which the commutative product is named `` forest product of multi-indices" as it is the counterpart to the forest product of rooted trees.
		A forest of multi-indices is a collection of multi-indices without order among them (a juxtaposition). The commutative forest product is denoted by ${\prod^{\bullet}}$ or $ \bullet$. The cardinal of the forest is defined as the number of individual multi-indices in this collection and it is denoted by $\mathrm{card}(\cdot)$. The identity element of the forest product is the empty forest $\emptyset$ with the forest cardinal of $0$.
  To further simplify the notation we use $z^{\bullet \tilde{\beta}}$ denoting a forest of multi-indices, where $\tilde\beta$ is a collection of multi-indices without order.  For $\tilde{\beta} = \{\beta_1,\ldots,\beta_n\}$
\begin{equs}
	z^{\bullet \tilde{\beta}} := 	\prod_{j=1}^{\bullet n} z^{\beta_j}.
\end{equs}
The repetition of individual populated multi-indices in $\tilde{\beta}$ is allowed, which means $z^{\beta_1},\ldots,z^{\beta_n}$ are not necessarily to be distinct.

To construct a graded algebra, we define the norm of a multi-index as the sum of each entry of its frequency multi-index,
\begin{equs}
|z^\beta|= |\beta| := \sum_{(i ,k) \in  [0,d] \times \mathbb{N}}  \beta (i ,k).
\end{equs}
This norm takes values in integers and is closely related to the Hölder regularity of signatures of rough paths through the later-defined ``$\gamma$-norm" \eqref{gammaNorm}. Thus it is suitable in grading associated multi-indices algebra. 
For $0 < \gamma < 1$, to describe the regularity of the rough path, we define a ``$\gamma$-norm" 
\begin{equs}\label{gammaNorm}
|z^\beta|_{\gamma}=	|\beta|_{\gamma} := \sum_{k \in   \mathbb{N}}  \frac{\beta (0 ,k)}{\gamma} + \sum_{(i ,k) \in  [1,d] \times \mathbb{N}}  \beta (i ,k).
\end{equs}

One can notice that we have symmetry in coefficients while Taylor expanding the solution. Therefore, to write a more compact series, we will use the  symmetry factor of a multi-index which is defined as the monomial of the factorial of arities
\begin{equs} \label{symmetry}
	S(z^\beta) := \prod_{(i ,k) \in  [0,d] \times \mathbb{N}} 
	(k!)^{\beta(i,k)}.
\end{equs}
The adjoint relation (details in \eqref{dual_Delta}) between algebra and its dual coalgebra is described by the inner product which is determined by the symmetry factor
\begin{equs} \label{eq:pairing}
	\langle z^\alpha , z^\beta \rangle := \delta_{\alpha,\beta} S(z^\alpha)
\end{equs}
where $\delta_{\alpha,\beta}$  is the Kronecker delta.
The values defined above can be generalised to forests of multi-indices as follows
\begin{equs}
	|\prod_{i=1}^{\bullet n} z^{\beta_i}| := \sum_{i=1}^n |z^{\beta_i}|, \quad 	|\prod_{i=1}^{\bullet n}z^{\beta_i}|_{\gamma} := \sum_{i=1}^n |z^{\beta_i}|_{\gamma}
\end{equs}
and for distinct $z^\beta_j$
\begin{equs}
	S\left(\prod_{j=1}^{\bullet m} \left(z^{\beta_j}\right)^{ \bullet r_j}\right) := \prod_{j=1}^m r_j! S\left(z^{\beta_j}\right)^{r_j}.
\end{equs}

In this paper, we only consider multi-indices that appear in the expansion of solutions, which satisfy the population condition defined below
\begin{equs}\label{population_condition}
	|\beta| - \sum_{(i ,k) \in  [0,d] \times \mathbb{N}} k \beta (i ,k) = 1\,.
\end{equs}
In the sequel we will call those multi-indices ``populated multi-indices" and the set of them is denoted by $\mathfrak{M}$. Furthermore we define $ \mathfrak{M}^{N}$, a subset of $\mathfrak{M}$, which consists of all the populated multi-indices with norm smaller than or equal to $N$. Accordingly, $\mathfrak{F}$ represents the set of the forests of populated multi-indices and $\mathfrak{F}^N$ is the subset $ \mathfrak{F}$, in which each forest has forest norm less than or equal to $N$.

 For $z^\alpha, z^\beta \in \mathfrak{M}$, let us define  $\triangleright: \mathfrak{M} \rightarrow \mathfrak{M}$
\begin{equs}	
	z^\beta \triangleright z^\alpha := z^\beta D z^\alpha ,
	\quad
	D:= \sum_{(i ,k) \in  [0,d] \times \mathbb{N}} z_{(i,k+1)}\partial_{z_{(i,k)}}\,,
\end{equs}
where $\partial_{z_{(i,k)}}$ is the directional derivative in the coordinate of ${z_{(i,k)}}$. ( i.e., for any $z^\alpha \in \mathfrak{M}$, if $\alpha(z_{(i,k)}) \ne 0$ then $ \partial_{z_{(i,k)}}z^\alpha = \frac{z^{\alpha}}{z_{(i,k)}} $. Otherwise, $ \partial_{z_{(i,k)}}z^\alpha = 0$.) The derivation of multi-indices defined below is a free Novikov product (see \cite{DL,Li23,BD23}) which is a pre-Lie product (see \eqref{pre-Lie}) satisfying the right-NAP identity (see \eqref{NAP}). One has for $z^\alpha, z^\beta, z^{\gamma} \in \mathfrak{M}$
	\begin{equs} \label{pre-Lie}
(z^\alpha \triangleright  z^\beta )  \triangleright z^{\gamma} - 		z^\alpha \triangleright  ( z^\beta   \triangleright z^{\gamma} ) &= (z^\beta \triangleright  z^\alpha )  \triangleright z^{\gamma} - 		z^\beta \triangleright  ( z^\alpha   \triangleright z^{\gamma} ),
	\\ \label{NAP}
	(z^\alpha \triangleright  z^\beta )  \triangleright z^{\gamma}  &= (z^\alpha \triangleright  z^\gamma )  \triangleright z^{\beta}.
	\end{equs}
Through the Guin-Oudom procedure \cite{Guin1,Guin2} one can construct a Grossman-Larson type associative product $\star$ which is isomorphic to the universal enveloping algebra of the underlying pre-Lie algebra. 
The first step of the procedure is to generalise the  domain of the product $\triangleright$ to 
$\mathrm{Span}_{\mathbb{R}}(\mathfrak{F}) \times \mathrm{Span}_{\mathbb{R}}(\mathfrak{F})$ such that the new product satisfies properties in \cite[Proposition 2.7]{Guin2}. Firstly, define this new product named ``simultaneous grafting" on the subspace  $\star_2: \mathfrak{F} \times \mathfrak{F} \rightarrow \mathrm{Span}_{\mathbb{R}}(\mathfrak{F})$ as the following. 
	\begin{equs} \label{definition_star_2}
		&\emptyset \star_2 \emptyset :=  \emptyset , \quad  
		\prod_{i=1}^{\bullet n} z^{\beta_i} \star_2  \emptyset = \emptyset \star_2 \prod_{i=1}^{\bullet n}z^{\beta_i}=\prod_{i=1}^{\bullet n} z^{\beta_i},
		\\&
		\prod_{i=1}^{\bullet n} z^{\beta_i} \star_2  z^{\bullet \tilde \alpha} := \prod_{i=1}^n z^{\beta_i} D^n z^{\bullet \tilde \alpha}
	\end{equs}
	for any $n \ge 1$ and $z^{\beta_i} \in \mathfrak{M}$ and $ z^{\bullet \tilde \alpha} \in \mathfrak{F} \setminus \emptyset$.
Here we require $\star_2$ to obey the Leibniz rule with respect to the forest product by setting 
\begin{equs}\label{Leibniz_D}
	D	z^{\bullet \tilde{\alpha}} = \sum_{j=1}^n   \left(\prod_{i\ne j}^{\bullet} z^{\alpha_i}\right)\, \bullet\, Dz^{\alpha_j}.
\end{equs}
Then, on the full space, $\star_2: \mathrm{Span}_{\mathbb{R}}(\mathfrak{F}) \times \mathrm{Span}_{\mathbb{R}}(\mathfrak{F}) \rightarrow \mathrm{Span}_{\mathbb{R}}(\mathfrak{F})$ is the linear extension of \eqref{definition_star_2}.

Notice that $\star_2$ is not associative. Therefore, the second step in the Guin--Oudom construction is to build an associative product $\star$ upon the $\star_2$ product by 
the deshuffle coproduct $\Delta_\shuffle : \mathfrak{F} \rightarrow \mathrm{Span}_{\mathbb{R}} (\mathfrak{F} \otimes \mathfrak{F})$, defined recursively as 
\begin{equs}
	& \Delta_\shuffle z^\beta = \emptyset \otimes z^\beta + z^\beta \otimes \emptyset, \quad \text{ for } z^\beta \in \mathfrak{M}
	\\&
	\Delta_\shuffle \prod_{i=1}^{\bullet n} z^{\beta_i} =  \prod_{i=1}^{\bullet n} \Delta_\shuffle  z^{\beta_i}
\end{equs}
where the forest product of two tensors are 
\begin{equs}
	(z^{\beta_1}\otimes z^{\beta_2})  \bullet (z^{\beta_3} \otimes z^{\beta_4}) 
	:= z^{\beta_1} {\bullet} z^{\beta_3} \otimes z^{\beta_2} {\bullet} z^{\beta_4}.
\end{equs}
We then take the linear extension of $\Delta_\shuffle$ and generalise it to $\Delta_\shuffle: \mathrm{Span}_{\mathbb{R}}(\mathfrak{F}) \times \mathrm{Span}_{\mathbb{R}}(\mathfrak{F}) \rightarrow \mathrm{Span}_{\mathbb{R}}(\mathfrak{F})$.
\begin{definition}[$\star$ product] \label{star_product}
The product	$\star: \mathfrak{F} \times \mathfrak{F} \rightarrow \mathrm{Span}_{\mathbb{R}}(\mathfrak{F})$ is defined through the $\star_2$ product and the deshuffle coproduct $\Delta_\shuffle$ in the following way. For any $n \ge 1$ and $z^{\beta_i} \in \mathfrak{M}$ and $ z^{\bullet \tilde \alpha} \in \mathfrak{F}\setminus \emptyset$,
	\begin{equs} \label{definition_star}
		\begin{aligned}
		&\emptyset \star \emptyset :=  \emptyset , \quad \prod_{i=1}^{\bullet n} z^{\beta_i} \star  \emptyset = \emptyset \star \prod_{i=1}^{\bullet n}z^{\beta_i}=\prod_{i=1}^{\bullet n} z^{\beta_i},
		\\&
		\prod_{i=1}^{\bullet n} z^{\beta_i} \star  z^{\bullet \tilde \alpha} := \mu  (\id  \otimes \cdot \star_2   z^{\bullet \tilde \alpha})  \Delta_\shuffle \left(\prod_{i=1}^{\bullet n}z^{\beta_i}\right)
		\end{aligned}
	\end{equs}
	where $\mu$ is the forest product which means for any $z^{\bullet \tilde \alpha}, z^{\bullet \tilde \beta} \in \mathfrak{F}$
	\begin{equs}
	\mu(z^{\bullet \tilde \alpha} \otimes z^{\bullet \tilde \beta}) = z^{\bullet \tilde \alpha}  \bullet z^{\bullet \tilde \beta}.
	\end{equs}
	Let us rewrite $\Delta_\shuffle (\CF) = \sum_{(\CF)_{\Delta_\shuffle}} \CF^{(1)} \otimes \CF^{(2)}$ for $\CF \in \mathfrak{F}$ and the sum runs over all possible splittings. Then for functions $f,g$ the composition in \eqref{definition_star} is defined as
	\begin{equs}
		\mu(f\otimes g)\Delta_\shuffle (\CF) :=  \sum_{(\CF)_{\Delta_\shuffle}} f(\CF^{(1)}) \bullet g(\CF^{(2)})\, .
	\end{equs}
Moreover, we can generalise $\star$ from $\mathfrak{F} \times \mathfrak{F}$ to  $ \mathrm{Span}_{\mathbb{R}}(\mathfrak{F}) \times \mathrm{Span}_{\mathbb{R}}(\mathfrak{F}) $ by taking the linear extension.
\end{definition}
Consequently, following the same reasoning of  \cite[Lemma 2.10]{Guin2} we can prove that $\star$ is associative and together with the deshuffle coproduct creates a Hopf algebra
$\mathcal{H}=(\mathrm{Span}_{\mathbb{R}}(\mathfrak{F}), \star, \Delta_\shuffle)$ which is graded as
\begin{equs}
	\CH = \bigoplus_{N=0}^\infty \CH_{(N)}
\end{equs}
where $\CH_{(N)}$ is the real vector space spanned by $\mathfrak{F}^N$ 
\begin{equs}
	\CH_{(N)} = \mathrm{Span}_{\mathbb{R}}(\{z^{\bullet \tilde \beta} \in \mathfrak{F}: | z^{\bullet \tilde \beta}| = N\})
\end{equs}
and we identify the space $\CH_{(0)}$ (the linear span of the empty forest) with real numbers.

The graded dual algebra of $\CH$ equips with a coproduct denoted by $ \Delta $ which is the adjoint of the product $\star$ under the pairing defined in \eqref{eq:pairing}
\begin{equs} \label{dual_Delta}
\langle	v\star w, u \rangle 
= \langle	v \otimes w, \Delta u \rangle,
\end{equs}
for any $v,w \in \CH$ and $u \in \mathrm{Span}_{\mathbb{R}}(\mathfrak{F})$.
The coproduct $\Delta$ is well defined and its explicit formula of $\Delta$ is given in \cite[Theorem 3.5]{BH24}. Therefore, one has  the dual Hopf algebra
$\CH^*=(\mathrm{Span}_{\mathbb{R}}(\mathfrak{F}), {\bullet}, \Delta)$, where we identify forests and dual forests with the same notation.

\subsection{Rough paths}
We use the explicit grading of the multi-index  Hopf algebra and its operations to define the notion of rough paths in this context by quickly recalling the usual theory of rough paths over a Hopf algebra, see \cite{Bel23,Tapia20, Manchon20}. For any given integer $N\geq 0$ we consider the  vector space  of truncated multi-indices forests
\[\mathcal{H}^N(\mathfrak{M})= \bigoplus_{n=0}^N \CH_{(n)} \]
obtained by simply quotienting the algebra $(\mathcal{H}, \star) $ with respect to the ideal generated by all multi-index monomials of homogeneity strictly bigger than $N$. The quotient product will be denoted  by $\star_N $.


Similarly,  inside $\mathcal{H}$, we introduce the subset of  truncated characters $\mathcal{G}^{N}(\mathfrak{M}) $ and truncated primitive elements $\mathfrak{g}^{N}(\mathfrak{M})$.
Suppose $\mathbf{x} \in \CH$ is in the form of
\begin{equs}
	\mathbf{x} = \sum_{u \in \mathfrak{F}}\frac{{\mathbf{x}}(u)}{S(u)}u
\end{equs}
where ${\mathbf{x}}(\cdot):  \mathrm{Span}_{\mathbb{R}}(\mathfrak{F}) \mapsto \mathbb{R}$ is linear in $\mathfrak{F}$.
Notice that $\mathbf{x}(u) = \langle\mathbf{x}, u \rangle$.
Then we say $\mathbf{x} \in \mathcal{G}^{N}(\mathfrak{M}) \subset \CH$ if it satisfies the following identity
\begin{equs}\label{group_property}
 \mathbf{x}(u  \bullet v) = \mathbf{x}(u)\mathbf{x}(v)
\end{equs}
for any $u, v \in \mathfrak{F}$ with $|u|+|v| \le N$.
On the other hand we say $\mathbf{x} \in \mathfrak{g}^{N}(\mathfrak{M})\subset \CH$ if $\mathbf{x} \in \CH^N(\mathfrak{M})$ and
\begin{equs}\label{Lie_property}
	\Delta_{\shuffle} \mathbf{x} = \mathbf{x} \otimes \emptyset + \emptyset \otimes \mathbf{x}
\end{equs}
which means $\mathbf{x}$ are the primitive elements. It is also helpful to check that, from the adjoint relation, we have the property that for any $\mathbf{x}_1, \mathbf{x}_2 \in \CH$ and $u \in \CH^*$
\begin{equs}\label{multipilication}
	 (\mathbf{x}_1 \star \mathbf{x}_2)(u)=
	\sum_{v,w \in \CH}\mathbf{x}_1(v)\mathbf{x}_2(w) \frac{\langle v\otimes w, \Delta u\rangle }{S(v)S(u)} = \mathbf{x}_1 (u^{(1)}) \mathbf{x}_2(u^{(2)})\,,
\end{equs}
where we use the Sweedler notation $\mathbf{x}_1 (u^{(1)}) \mathbf{x}_2(u^{(2)})$ for the coproduct $\Delta$.

Using standard properties of connected graded Hopf algebras, we have that  
$(\mathcal{G}^{N}(\mathfrak{M}), \star_{N})$ is  a finite-dimensional Lie group with Lie algebra explicitly given by $\mathfrak{g}^{N}(\mathfrak{M})$ and Lie operation given by the commutator of $\star_{N}$.  Moreover, in this setting  the usual exponential map of the Lie group turns out to be a bijection which can be explicitly written via the maps   $\exp_{\star_N}\colon \mathcal{H}^N(\mathfrak{M})\to \mathcal{H}^N(\mathfrak{M})$,  $\log_{\star_N}\colon \mathcal{H}^N(\mathfrak{M})\to \mathcal{H}^N(\mathfrak{M})$ defined by
 \begin{equation}\label{eq_exp_Hopf}
 \exp^{\star}_{N}{\mathbf{x}}= \sum_{n\geq 0}^{N}\frac{\mathbf{x}^{\star_N n}}{n!}\,,\quad \log^{\star}_{N}{\bf x}= \sum_{n\geq 1}^{N}(-1)^{n+1}\frac{(\mathbf{x}-\emptyset)^{\star_N n}}{n}.
 \end{equation}

The notion of rough path then is defined starting from two parameters paths defined from the $2$ simplex $\Delta^2_{T}=\{(s,t)\in [0,T]^2\colon s\leq t\}$.
\begin{definition}  \label{dfn:genRP}
Let $\gamma\in (0,1)$ and $N_{\gamma}=\lfloor \gamma^{-1}\rfloor$. A \emph{multi-index rough path} of (H\"older) regularity $\gamma$ is a map $\mbX\colon\Delta_T^2\to \mathcal{G}^{N_{\gamma}}(\mathfrak{M})$ satisfying these properties:
\begin{itemize}

\item for any $s,u,t\in[0,T]$ such that $s\leq u\leq t$ one has 
\begin{equation}\label{eq:chen}
\mbX_{s,u}\star_{N_{\gamma}} \mbX_{u,t}=\mbX_{s,t}\,;
\end{equation}
\item for all $z^{\beta}\in \mathfrak{M}^{N_{\gamma}}$ 
\begin{equation}\label{eq:genrpbound}
\sup_{s\neq t\in \Delta_T^2}\frac{\vert \mbX_{s,t}(z^{\beta})\vert}{|t-s|^{|\beta|_{\gamma}}}<+\infty\,.
\end{equation}
\end{itemize}
\end{definition}
\begin{remark}
Our definition of a multi-index rough path is equivalent to the one given in \cite[Definition 3.10]{Li23}. While the formulations may appear different at first glance, they ultimately describe the same algebraic and analytic structure, ensuring consistency across approaches. Indeed in that case the author formulates the identity \eqref{eq:chen} in terms of of the universal enveloping algebra $U(\mathfrak{M})$ and its action on $(\mathrm{Span}_{\mathbb{R}}(\mathfrak{F}), {\bullet}, \Delta)$ which is isomorphic to $\mathcal{H}$.
\end{remark}
Since by construction any given  multi-index rough path $\mbX$ of regularity $\gamma$ is an element of $\mathcal{G}^{N_{\gamma}}(\mathfrak{M})$, we can compute the logarithm of  $\mbX$  and we  will use the shorthand notation
\begin{equs}
	\label{logarithm_rough}
	\boldsymbol\Lambda^{\mbX}_{s,t}= \log_{\star_{{N}_\gamma}}(\mbX_{s,t})\,.
\end{equs}
Moreover, from the multiplicative property one has immediately 
\[\sup_{s\neq t\in \Delta_T^2}\frac{\vert \mbX_{s,t}(u)\vert}{|t-s|^{|u|_{\gamma}}}<\infty \]
for any $u\in \mathfrak{F}\colon |u|\leq N$. For any given map $\mbZ\colon \Delta_T^2\to \mathcal{H}^N(\mathfrak{M})$ we define the  quantity
\begin{equation}\label{defn_Norm}
\Vert \mbZ\Vert= \max_{u\in \mathfrak{F}(\mathfrak{M})\colon |u|\leq N} \left(\sup_{s\neq t\in \Delta_T^2}\frac{\vert \mbZ_{s,t}(u)\vert}{|t-s|^{|u|_{\gamma}}}\right)^{\frac{1}{|u|}}\,.
\end{equation}
 Thanks to Definition \ref{dfn:genRP} one has $\Vert \mbX\Vert<\infty$ for any multi-index rough path $\mbX$ by simply combining \eqref{eq:genrpbound} with the group property \eqref{group_property}. The factor $\frac{1}{|u|}$ inside \eqref{defn_Norm} allows to compare in an equivalent way a rough path and its logarithm, see \cite[Lemma 3.15]{KL23} for a proof of the following estimate.
\begin{proposition}\label{prop_estimate}
For any  $\gamma\in (0,1)$ and $d\geq 1$ integer there exists two constants $c, C>0$ depending only on $\gamma$ and $d$ such that for any  multi-index rough path $\mathbf{X}$ of  regularity $\gamma$  one has
\[c\Vert \mbX\Vert\leq \Vert \Lambda^{\mbX}\Vert\leq C\Vert \mbX\Vert.\]

\end{proposition}
By simply applying the property \eqref{Lie_property} one has that 
$ \Lambda^{\mbX}_{s,t}( \prod_{i=1}^{\bullet n} z^{\beta_i} )=0$ for any integer  $n\geq 2$ leaving a non-trivial value of $\Lambda^{\mbX}$ only over the set $\mathfrak{M}$. 

As with all other families of rough paths,  any multi-index rough path of regularity $\gamma\in (0,1) $ extends uniquely to a path $\mbX'\colon\Delta_T^2\to \CH'$ with $\CH'$ the  stands for the space of continuous $\mathbb{R}$ linear forms over $\mathrm{Span}_{\mathbb{R}}(\mathfrak{F})$, satisfying the properties \eqref{eq:chen}, \eqref{group_property} without any truncation and property \eqref{eq:genrpbound} for any $z^{\beta}\in \mathfrak{M}$, see \cite[Proposition 3.13]{Li23}.  In what follows, we will always identify  any multi-index rough path $\mbX$ with its extension $\mbX'$ using the same notation so that  the expression $\mbX_{s,t}(u)$ is  always rigorously defined even when $u>N_{\gamma}$.

\subsection{Elementary differentials and pseudo-bialgebra property}
\label{sec::2.3}
In this section we will introduce elementary differentials from the Taylor coefficients in the expansion of the solution. Since we now focus on the algebraic properties, it suffices to consider $f \in \CC^{\infty}(\mathbb{R}, \mathbb{R}^d)$ and  $f_0\in \CC^{\infty}$. The requirement of smoothness of $f$ will be lowered later in  section \ref{Sec::3}.
\begin{definition}[elementary differentials]
	For any populated multi-index $z^\beta \in \mathfrak{M}$ and $y \in \mathbb{R}$, the elementary differential is  given by
	\begin{equs} \label{elementary_differential_ODE}
		\Upsilon_f[z^\beta](y) := \prod_{(i ,k) \in  [0,d] \times \mathbb{N}} \left( f_i^{(k)}(y)\right)^{\beta(i,k)} 
	\end{equs}
	where  $f_i^{(k)}$ is the $k$-th order ordinary derivative of $f_i$. Moreover, we extend the domain of $\Upsilon_f[\cdot]$ to forests, by letting it multiplicative with respect to the forest product of multi-indices, i.e.,
\begin{equs}
	\Upsilon_f[ \prod_{i=1}^{\bullet n} z^{\beta_i}](y):= \prod_{i=1}^n \Upsilon_f[ z^{\beta_i}](y)\, .
\end{equs}
\end{definition}

 In addition to this functional representation $\Upsilon_f[\cdot]\colon  \mathfrak{M} \to \CC^\infty$, we will do a small abuse of notation and  use  $\Upsilon_f[\cdot]$ and its functions to describe a map $\Upsilon_f[\cdot]\colon  \mathfrak{M} \to \mathrm{Vec}(\mathbb{R})$ with $\mathrm{Vec}(\mathbb{R})$  the space of vector fields over $\mathbb{R}$ generated by  $f$, that is the free algebra of operator over $\CC^\infty$ generated by composition $\circ$ of vector fields starting from the  elementary derivations $\Upsilon_f[z_{(i,0)}](\cdot)(y)\colon \CC^\infty\to \CC^\infty$ for $i\in [0,d]$, defined by 
 \[  \Upsilon_f[z_{(i,0)}](\psi)(y)
 = f_i(y)\psi'(y)\,.\]
 \begin{definition}[elementary vector fields]
For any $z^\beta \in \mathfrak{M}$, we define the elementary vector field	$\Upsilon_f[ z^{\beta}] \colon  \CC^\infty \to \CC^\infty$ on any $\psi\in  \CC^\infty$ as the function
\begin{equs}\label{identification_1}
	\Upsilon_f[ z^{\beta}] (\psi)(y) := \Upsilon_f[ z^{\beta}](y) \psi'(y)
\end{equs}
  For generic multi-index forests of multi-indices we set
\begin{equs}\label{identification_2}
	&\Upsilon_f[\emptyset ] (\psi)(y) := \psi(y),
	\\&
	\Upsilon_f[\prod_{i=1}^{\bullet n} z^{\beta_i}] (\psi)(y) := \prod_{i=1}^n\Upsilon_f[ z^{\beta_i}](y)  \psi^{(n)}(y)\,.
\end{equs}
\end{definition}

One can immediately remark that the identity \eqref{identification_1} becomes  \eqref{elementary_differential_ODE} when $\psi = \id$. Both elementary differentials and  elementary vector fields describe the same functions but under different interpretations,  with the second one coming from differential geometry.  In the sequel, we will  use the double bracket notation $\Upsilon_f[\cdot] (\cdot)(\cdot)$ when we consider $\Upsilon_f[\cdot]$ as an element of $\mathrm{Vec}(\mathbb{R})$ and the simpler $\Upsilon_f[\cdot](\cdot)$ when we consider $\Upsilon_f[\cdot]$ as a function.
	As an example, one has
	\begin{equs}
	 \Upsilon_f[z_{(i,k)}](\psi)(y) & = f_i^{(k)}(y) \psi'(y),
	 \\
	 \Upsilon_f[z_{(i,k)}](\id)(y) & =\Upsilon_f[z_{(i,k)}](y) =  f^{(k)}_i(y).
	\end{equs}
	 One can observe that the elementary differential has the following morphism property, which connects the expansions of solutions to the combinatorial objects.
\begin{proposition} \label{morphism_element}
	 The elementary differential is a homomorphism when one defines the product between elementary differentials as the composition, which means for any $u, v \in \CH$ and $y \in \mathbb{R}$
\begin{equs}
		\Upsilon_f[u \star v](\psi) (y)
			= 
			\Upsilon_f[u]  \circ \Upsilon_f[v](\psi) (y)
\end{equs}
where $\circ$ stands for the composition of elementary vector fields in the second input, i.e., $\Upsilon_f[u]  \circ \Upsilon_f[v](\psi) (\cdot) = \Upsilon_f[u]   \left(\Upsilon_f[v](\psi)\right)(\cdot)$.
\end{proposition}
\begin{proof}
	We start with proving that the elementary differential is a pre-Lie morphism. 
	\begin{equs}
		\Upsilon_f[ z^{\beta} \triangleright z^\alpha](\psi)(y) &=  
		\Upsilon_f\left[ z^{\beta} \sum_{(i ,k) \in   [0,d] \times \mathbb{N}} z_{(i,k+1)}\partial_{z_{(i,k)}}  z^\alpha \right](y) \, \psi'(y)
		\\&=
		\Upsilon_f[ z^{\beta}](y) \sum_{(i ,k) \in   [0,d] \times \mathbb{N}} \Upsilon_f\left[ z_{(i,k+1)}\partial_{z_{(i,k)}}  z^\alpha \right](y)  \, \psi'(y).
	\end{equs} From the definition of the elementary differential $ \Upsilon_f[\cdot](\cdot) $, one can see it is multiplicative with respect to the multi-index product. Then 
	\begin{equs}
		\sum_{(i ,k) \in   [0,d] \times \mathbb{N}} \Upsilon_f\left[ z_{(i,k+1)}\partial_{z_{(i,k)}}  z^\alpha \right] (y) & = \sum_{\substack {(i ,k) \in   [0,d] \times \mathbb{N} \\ \alpha(i ,k) \ne 0}} \frac{f_i^{(k+1)} \Upsilon_f[ z^{\alpha}]}{f_i^{(k)}}(y) \\ &  = \partial \Upsilon_f[ z^{\alpha}](y) = \Upsilon_f'[ z^{\alpha}](y),
	\end{equs}
where $\partial$ denote the derivative in $y$. One has 
	\begin{equs}
		\Upsilon_f[ z^{\beta} \triangleright z^\alpha](\psi)(y)& = \Upsilon_f[ z^{\beta}](y) \Upsilon_f'[ z^{\alpha}] (y) \psi'(y)
		\\ &= 
	 \left(\Upsilon_f[ z^{\beta}]   \left(\Upsilon_f[z^{\alpha}]\right)\right)(\psi)(y).
	\end{equs}
	Consequently, one can get the composition by
	\begin{equs}
		\Upsilon_f[ z^{\beta}] \circ \Upsilon_f[z^\alpha](\psi)(y) 
		=\Upsilon_f[ z^{\beta} \triangleright z^\alpha](\psi)(y) +  \Upsilon_f[ z^{\beta} \bullet z^\alpha](\psi)(y).
		\end{equs} 
	Furthermore, by the universal property 
	between the pre-Lie product and its universal enveloping algebra, we can generalise the equality to  $\star$:
 by the Leibniz rule
\begin{equs}
		\Upsilon_f[ \prod_{i=1}^{\bullet n}  z^{\beta_i} \star  z^{\bullet\tilde \alpha}] (\psi)(y) 
		&= \sum_{I \sqcup J = \{1,\ldots,n\}}
		 \Upsilon_f\left[\prod_{i \in I} ^{\bullet }  z^{\beta_i} \bullet \left(\prod_{j \in J} ^{\bullet }z^{\beta_j} \star_2 z^{\bullet \tilde \alpha}\right)\right](\psi)(y)
		\\&=\sum_{I \sqcup J = \{1,\ldots,n\}}  
		\prod_{i=1}^n\Upsilon_f[z^{\beta_i}] (y)
		\Upsilon_f^{(|J|)}[  z^{\bullet\tilde \alpha}](y)
		\psi^{(|I|+\mathrm{card}(\tilde z^{\tilde \alpha}))}(y)
		\\&=
		\prod_{i=1}^n\Upsilon_f[z^{\beta_i}](y) (\prod_{l=1}^{\mathrm{card}(\tilde z^{\tilde \alpha})}\Upsilon_f[z^{ \alpha_l}]	\psi^{(\mathrm{card}( z^{\tilde \alpha}))})^{(n)}(y)
		\\&=
		\Upsilon_f[ \prod_{i=1}^{\bullet n}  z^{\beta_i}] \circ \Upsilon_f[ z^{\bullet \tilde \alpha}](\psi)(y).
	\end{equs} 
		where $\sum_{I \sqcup J = \{1,\ldots,n\}}$ runs over all possible ways to split the index set $\{1,\ldots,n\}$ into to two pieces $I$ and $J$ (they can be empty set) and $|\cdot|$ is the cardinal of the set.
	Finally, the result can be generalised to any element $u,v \in \CH$ as $\star$ and elementary differentials are linear in forests.
\end{proof}
For $u, v \in \CH$ and $ \phi, \psi \in \mathcal{C}^{\infty} $, we set 
\begin{equs}
	\Upsilon_f[ u \otimes v] (\phi \otimes \psi)(\cdot) = 	\Upsilon_f[ u] (\phi ) (\cdot)\Upsilon_f[ v] ( \psi) (\cdot).
\end{equs}
\begin{proposition} \label{Leibniz_identity}
	For any $u \in \CH$ and $\phi, \psi \in \CC^\infty$,  we have 
	\begin{equs}\label{leibniz_id}
		\Upsilon_f[ \Delta_\shuffle u] (\phi \otimes \psi) = \Upsilon_f[u] (\phi \psi).
	\end{equs}
\end{proposition}
\begin{proof}
	We only have to prove the identity for the primitive element $z^\beta \in \mathfrak{M}$, which is a direct result of the Leibniz rule. On the left-hand side 
	\begin{equs}
		\Upsilon_f[ \Delta_\shuffle z^{\beta}] (\phi \otimes \psi)(y) &= \Upsilon_f[  z^{\beta} \otimes \emptyset] (\phi \otimes \psi)(y) + \Upsilon_f[ \emptyset \otimes  z^{\beta}] (\phi \otimes \psi)(y)
		\\ & = \psi(y) \Upsilon_f[z^{\beta}](y)  \phi'(y)  + \phi(y) \Upsilon_f[z^{\beta}](y)  \psi'(y)
	\end{equs}
	since $\Upsilon_f[\emptyset] (\phi) = \phi$. On the right-hand side, according to the definition of elementary differentials composing with smooth functions and the Leibniz rule,
	\begin{equs}
		\Upsilon_f[z^{\beta}] (\phi \psi)(y) &= \Upsilon_f[z^{\beta}](y)  (\phi \psi)'(y)
		\\&= \psi(y) \Upsilon_f[z^{\beta}](y) \phi'(y)  + \phi(y) \Upsilon_f[z^{\beta}](y) \psi'(y).
	\end{equs}
	
	Since the primitive elements are ``generators" (one can prove by induction on the cardinal of the forest) and both $\Delta_{\shuffle}$ and elementary differentials are linear in forests of multi-indices, the result can be generalised to any $u \in \CH$.
\end{proof}

\begin{remark}
It is interesting to note that
\begin{equs}
	\Upsilon_f[ \prod_{i=1}^{\bullet n} z^{\beta_i}] (\id)(\cdot) = 0
\end{equs}
for any $n>1$. Therefore for the map $\mathbf{x}\mapsto\Upsilon_f[ \mathbf{x}] (\id)$ defined over $\mathcal{H}^N(\mathfrak{M})\setminus \{\emptyset\}$ acts as a projection on  $ \mathfrak{g}(\mathfrak{M})^N$ for any integer $N\geq 1$. Hence the need of property \eqref{leibniz_id}.

\end{remark}

\section{Solution theory for a multi-index RDE}\label{Sec::3}
We show now how to use the properties of multi-index rough paths and elementary differentials over multi-index to build local and global flow solution for the equation \eqref{RDE} according to the flow approach developed  in \cite{B15,B19} and extended in \cite{KL23} for  rough paths over a Hopf algebra. All through the section we will use the notations $\Delta^2_{T}=\{(s,t)\in [0,T]^2\colon s\leq t\}$, $\Delta_{T}=\{(s,t)\in [0,T]^2\colon s=t\}$ for some generic parameter $T>0$ chosen freely.

\subsection{Solution flow for multi-index RDE}
Instead of formulating \eqref{RDE} as a Cauchy problem we want to state it in terms of a non-autonomous flow with state space $\mathbb{R}$ so that the solution of \eqref{RDE} at some time $t\in [0,T]$ starting from $y_0\in \mathbb{R}$ at time $s\leq t $ will be denoted by $\varphi_{s,t}(y_0)$.

More generally, we say that a map $\varphi\colon \Delta_T^2\times \mathbb{R}\rightarrow \mathbb{R}$,  is a global real flow if it satisfies for any $s\leq r\leq t$
\begin{equation}\label{flow_property}
\varphi_{s,t} = \varphi_{r,t}\circ\varphi_{s,r} \,.
\end{equation}
To take in account phenomena of local existence and uniqueness we introduce the concept of local flow from \cite{KL23}. 
\begin{definition} An open set $O$ such that  $\Delta_T\times \mathbb{R}\subset O \subset \Delta_T^2\times \mathbb{R}$ is said to be an admissible domain  if it satisfies the following properties:
\begin{itemize}
\item for any compact set $K\subset \mathbb{R}$, there exists a value $0<T(K)\le T$ (called non-exploding time) such that $((s,t),x)\in O$ for each $x\in K$ and $|t-s|\le T(K)$;
\item for any $((s,t),x)\in O$ and $s\le r\le t$,  $((s,r),x)\in O$.
\end{itemize}
For any given admissible set $O$ we call any map $\varphi\colon O\rightarrow \mathbb{R}$ a \textbf{local real flow} if it satisfies \eqref{flow_property}   for any  $((s, t), x) \in O$ and $s \leq r \leq t$ such that $((r, t), \varphi_{s,r}(x))\in O$ .  
\end{definition}

As pointed out in the introduction, by performing a formal Taylor expansion of how the solution of \eqref{RDE} should look like, we want to use the approximate expansion \eqref{Butcher_expansion}
for any given multi-index rough path  $\mbX$ of regularity $\gamma\in (0,1)$ and any couple of vector fields $f\in \CC^{\alpha_1}(\mathbb{R}, \mathbb{R}^d  )$, $f_0 \in \mathcal{C}^{\alpha_2}$  for some $\alpha_1>N_{\gamma}-1$ and $ \alpha_2>0$ to actually construct a proper solution.

Using the identification between functions and real vector fields generated by $f$,  we can rewrite for every $y \in \mathbb{R}$  it as 
\begin{equation}\label{approximate_exp2}
\sum_{\substack{ z^{\beta}\in \mathfrak{M}\\0\leq \vert z^{\beta}\vert_{\gamma}\leq N_{\gamma}}}\frac{\Upsilon_f[z^{\beta}](\id)(y)}{S(z^{\beta})} \mbX_{s,t}(z^{\beta}) \,.
\end{equation}
 Replacing now the notion of solution with a flow and localising the values on a compact set we can finally state what is flow solution.

\begin{definition}[Flow solution]
\label{DefnGeneralRDESolution}
 Let  $\mbX$ be a multi-index rough path of regularity $\gamma$ and $f\in \CC^{\alpha_1}(\mathbb{R}, \mathbb{R}^d) $, $f_0 \in \mathcal{C}^{\alpha_2}$  vector fields with  $\alpha_1>N_{\gamma}-1$ and $ \alpha_2>0$. A local real flow $\phi\colon O\to \mathbb{R}$ is said to be a local flow solution of the rough differential equation \eqref{RDE} if there exists a constant $a>1$ independent of $\mbX$ such that for any compact set $K\subset \mathbb{R}$ there exists a positive constant  $c>0$ dependent possibly on $\mbX$  and $K$ such that
\begin{equation}
\label{DefnSolRDEGeneral}
\sup_{y\in K}|\varphi_{s,t}(y)- \sum_{\substack{ z^{\beta}\in \mathfrak{M}\\0\leq \vert z^{\beta}\vert_{\gamma}\leq N_{\gamma}}}\frac{\Upsilon_f[z^{\beta}](\id)(y)}{S(z^{\beta})} \mbX_{s,t}(z^{\beta}) | \leq c\,|t-s|^a
\end{equation}
for all $ s\leq t$ such that $|t-s|\leq T(K) $. If $O=\Delta_T^2\times \mathbb{R}$ with $T(K) =T$ and the estimate \eqref{DefnSolRDEGeneral} is uniform on $K$ we say that  $\phi$ is a global flow solution.
\end{definition}
\begin{remark}\label{remark_simplified}
By looking at the expression \eqref{approximate_exp2}, one realises that
	soon as the multi-index $ z^{\beta} $ contains $z_{(0,k)}$ and extra variables, one has $  |z^{\beta}|_\gamma > N_{\gamma} $. Therefore, expression  \eqref{approximate_exp2} simplifies to 
	\begin{equs}
(t-s) f_0(y)	+	\sum_{\substack{ z^{\beta}\in \mathfrak{M}_0\\0\leq \vert z^{\beta}\vert_\gamma \leq N_{\gamma}}}\frac{\Upsilon_f[z^{\beta}](\id)(y)}{S(z^{\beta})} \mbX_{s,t}( z^{\beta})
	\end{equs}
where $ \mathfrak{M}_0 $ is the set of populated multi-indices in which the frequencies of variable $z_{(0,k)}$ are $0$ for every $k \in \mathbb{N}$. This is enough for the well-posedness to consider a multi-index rough path defined only on these multi-indices. In  Section \ref{Sec::4}, we consider the translation of such a rough path and the correct framework for understanding it is to use this extension. Therefore, we have decided to work directly in this framework from the very beginning.
	\end{remark}

In case of a rough differential equation the construction of a flow is provided by introducing a proper approximation defined using the logarithm of a rough path.


\begin{definition} \label{def_loga_flow}
Let $\mbX$ be a multi-index rough path  of  regularity $\gamma\in (0,1)$ and  $f\in \CC^{\alpha_1}(\mathbb{R}, \mathbb{R}^d  )$, $f_0 \in \mathcal{C}^{\alpha_2}$  for some $\alpha_1>N_{\gamma}$ and $ \alpha_2>1$. For any $y\in \mathbb{R}$ we denote by $O^{\mbX}(y)$ the biggest  subset of $\Delta_T^2$ such that for any $(s,t)\in O(y)$ the following Cauchy problem 
\begin{equation}\label{def_logode}
\left\{\begin{aligned}
&\dot{Z}_r=  \sum_{\substack{ z^{\beta}\in \mathfrak{M}\\1\leq \vert z^{\beta}\vert_{\gamma}\leq N_{\gamma}}}\frac{\Upsilon_f[z^{\beta}](Z_r)}{S(z^{\beta})} \Lambda^{\mbX}_{s,t}( z^\beta)\\&Z_0=y\,
\end{aligned}\right.
\end{equation}
admits a solution  $Z$ defined over $[0,1]$.  Introducing the set 
\[O^{\mbX}= \bigcup_{y\in \mathbb{R}}O^{\mbX}(y)\times \{y\}\subset \Delta_T^2\times \mathbb{R}\]
we call the map $\mu^{\mbX}\colon O^{\mbX}\to \mathbb{R}$ given by $\mu^{\mbX}_{s,t}(y)=Z_1$ the log-ODE almost flow.
\end{definition}

Since the elementary differential $\Upsilon_f[z^{\beta}]$ involves at most $N_{\gamma}-1$ derivative of the vector field $f$ and $\alpha_1>N_{\gamma}$ and in the expression \eqref{def_logode} there is only $f_0$ (no derivatives of $f_0$ involved), see also Remark \ref{remark_simplified}, the ordinary differential equation \eqref{def_logode} involves a space time non-linearity in space at least of class $C^1$ in the space variable therefore locally Lipschitz. For this reason  we can apply all standard result of Cauchy problem with a Local Lipschitz non-linearity obtaining immediately that $O^{\mbX}$ is an admissible set with non-exploding times $T^{\mbX}(K)$ for any compact $K\subset \mathbb{R}$. Moreover, in case  $f\in \CC^{\alpha_1}_b(\mathbb{R}, \mathbb{R}^d  )$, $f_0 \in \mathcal{C}^{\alpha_2}_b$  or they are linear, the differential equation \eqref{def_logode} admits a global solution and one can take $O^{\mbX}= \Delta_T^2\times \mathbb{R}$ with $T^{\mbX}(K)=T$ uniformly over $K$.

The choice of using the logarithm can be related to  the properties in section \ref{Sec::2} and can be used to naturally relate $\mu^{\mbX}$ with the expansion \eqref{approximate_exp2}. 
\begin{proposition}\label{main_prop}
Let $\mbX$ be a multi-index rough path of  regularity $\gamma\in (0,1)$ and  $f\in \CC^{\alpha_1}(\mathbb{R}, \mathbb{R}^d) $, $f_0 \in \mathcal{C}^{\alpha_2}$  vector fields with  $\alpha_1>N_{\gamma}$ and $ \alpha_2>1$.  For any compact set $K\subset \mathbb{R}$ and $\psi\in \mathcal{C}^{\alpha_3}(\mathbb{R})$ with  $\alpha_3>N_{\gamma}+1$ there exists a constant $C>0$ such that one has 
\begin{equation}\label{taylor_exp_flow}
\sup_{y\in K}|\psi\circ \mu^{\mbX}_{s,t}(y)- \sum_{\substack{u\in \mathfrak{F}(\mathfrak{M})\\0\leq \vert u\vert_{\gamma}\leq N_{\gamma}}}\frac{\Upsilon_f[u](\psi)(y)}{S(u)} \mbX_{s,t}(u) | \leq C\,|t-s|^{(N_{\gamma} +1)\gamma}
\end{equation}
for any couple $s\leq t$ such that $|t-s|<T^{\mbX}(K)$. In addition one has the local Lipschitz estimate
\begin{equation}\label{taylor_exp_flow_2}
\begin{split}
&\sup_{z,w\in K}|\psi\circ \mu^{\mbX}_{s,t}(y)-\psi\circ \mu^{\mbX}_{s,t}(w) -\sum_{\substack{u\in \mathfrak{F}(\mathfrak{M})\\0\leq \vert u\vert_{\gamma}\leq N_{\gamma}}}\frac{\Upsilon_f[u](\psi)(y)-\Upsilon_f[u](\psi)(w)}{S(u)} \mbX_{s,t}(u)|\\& \leq C\,|t-s|^{(N_{\gamma} +1)\gamma}|y-w|\,.
\end{split}
\end{equation}
Moreover, in case  $f\in \CC^{\alpha_1}_b(\mathbb{R}, \mathbb{R}^d) $, $f_0 \in \mathcal{C}^{\alpha_2}_b$   or linear and $\psi\in  \CC^{\alpha_3}_b $ all these properties hold independently on the compact set $K$.
\end{proposition}

\begin{proof}
Let  $K\subset \mathbb{R}$ be a fixed compact subset. Most of this proof follows from \cite[Lemma 4.7]{KL23}, which we extend to the setting of multi-index.  We will also use the shorthand notation $N_{\gamma}=N$.  Using the explicit form of \eqref{def_logode} and the action of $\Upsilon_f$ on smooth function, we can write  for any couple $s\leq t$ such that $|t-s|<T^{\mbX}(K)$ and any $y\in K$ the simple consequence of the fundamental theory of calculus
\begin{equation}\label{first_expansion}
\begin{split}
\psi\circ \mu^{\mbX}_{s,t}(y)&=  \psi(y)+ \int_0^{1}\sum_{\substack{ z^{\beta}\in \mathfrak{M}\\1\leq \vert z^{\beta}\vert_{\gamma}\leq N}}\frac{\Upsilon_f[z^{\beta}](Z_r)}{S(z^{\beta})}\Lambda^{\mbX}_{s,t}( z^\beta) \psi'(Z_r) \mathrm{d} r\\ &=\psi(y)+ \int_0^{1}\sum_{\substack{ z^{\beta}\in \mathfrak{M}\\1\leq \vert z^{\beta}\vert_{\gamma}\leq N}}\frac{\Upsilon_f[z^{\beta}](\psi)(Z_r)}{S(z^{\beta})} \Lambda^{\mbX}_{s,t}(z^\beta) \mathrm{d} r\,.
\end{split}
\end{equation}
Using the assumptions   $f\in \CC^{\alpha_1}(\mathbb{R}, \mathbb{R}^d) $, $f_0 \in \mathcal{C}^{\alpha_2}$ and $\psi\in \CC^{\alpha_3}$ we iterate the identity  in the following way: for any multi-index $z^{\beta}$ with $|\beta|_{\gamma}=k$ we apply again the fundamental theory of calculus $ N-k+1 $ times over all the functions with at most $N-1$ derivatives over $f$, thereby   obtaining for any integer $2\leq k\leq N+1$ a subset $A_k\subset \{\beta_1\,, \ldots \,,\beta_k\in (\mathfrak{M}^{N})^k\colon  |\beta_1 |_{\gamma}+\cdots +|\beta_k |_{\gamma} > N\}$ such that
\begin{equs}
\, &\psi\circ \mu^{\mbX}_{s,t}(y)=  \sum_{k=0}^{N}\frac{1}{k!}\sum_{\substack{ \beta_1\,, \ldots \,,\beta_k\in \mathfrak{M}^{N}\\ |\beta_1 |_{\gamma}+\cdots +|\beta_k |_{\gamma} \leq N}}  \left(\frac{\Upsilon_f[z^{\beta_1}]}{S(z^{\beta_1})}\circ \cdots \circ \frac{\Upsilon_f[z^{\beta_k}]}{S(z^{\beta_k})}\right)(\psi)(y)\prod_{l=1}^k \Lambda^{\mbX}_{s,t}( z^{\beta_l})\\&+ \sum_{k=2}^{N+1}\int_{\Delta^{k}}\sum_{ \beta_1\,, \ldots ,\beta_{k}\in A_k} \left(\frac{\Upsilon_f[z^{\beta_1}]}{S(z^{\beta_1})}\circ \cdots \circ \frac{\Upsilon_f[z^{\beta_{k}}]}{S(z^{\beta_{k}})}\right)(\psi)(Z_{r_{k}})\prod_{l=1}^{k} \Lambda^{\mbX}_{s,t}(z^{\beta_l}) \mathrm{d} r_{k}\ldots \mathrm{d} r_1\,,
\end{equs}
where $\Delta^{k}$ is the unit simplex of dimension $k$. It is important to remark that the sets $ A_k$  are chosen such that expression make sense for any   $f\in \CC^{\alpha_1}(\mathbb{R}, \mathbb{R}^d) $, $f_0 \in \mathcal{C}^{\alpha_2}$ and $\psi\in \CC^{\alpha_3}$. We can then apply Proposition \ref{morphism_element}  to obtain 
\begin{equation}\label{eq_total_expansion}
\begin{split}
&\psi\circ \mu^{\mbX}_{s,t}(y)=  \sum_{k=0}^{N}\frac{1}{k!}\sum_{\substack{ \beta_1\,, \ldots \,,\beta_k\in \mathfrak{M}^{N}\\ |\beta_1 |_{\gamma}+\cdots +|\beta_k |_{\gamma} \leq N}} \Upsilon_f[z^{\beta_1}\star  \cdots \star z^{\beta_k}](\psi)(y)\prod_{l=1}^k \frac{\Lambda^{\mbX}_{s,t}( z^{\beta_l})}{S(z^{\beta_l})}\\&+\sum_{k=2}^{N+1}\int_{\Delta^{k}}\sum_{ \beta_1\,, \ldots ,\beta_{k}\in A_k} \left(\frac{\Upsilon_f[z^{\beta_1}\star\cdots\star z^{\beta_{k}} ]}{S(z^{\beta_1})\cdots S(z^{\beta_{k}})} \right)(\psi)(Z_{r_{k}})\prod_{l=1}^{k} \Lambda^{\mbX}_{s,t}(z^{\beta_l}) \mathrm{d} r_{k}\ldots \mathrm{d} r_1\,.
\end{split}
\end{equation}
To finally understand the arise of the rough path $\mbX$ in \eqref{taylor_exp_flow} we remark that one has from the duality \eqref{dual_Delta} and \eqref{multipilication} that
\begin{equation}\label{eq_partial_expansion}
\begin{split}
&\sum_{\substack{ \beta_1\,, \ldots \,,\beta_k\in \mathfrak{M}^{N}\\ |\beta_1 |_{\gamma}+\cdots +|\beta_k |_{\gamma} \leq N}} \Upsilon_f[z^{\beta_1}\star  \cdots \star z^{\beta_k}](\psi)(y)\prod_{l=1}^k \frac{\Lambda^{\mbX}_{s,t}( z^{\beta_l})}{S(z^{\beta_l})}\\&= \sum_{\substack{u\in \mathfrak{F}(\mathfrak{M})\\0\leq \vert u\vert_{\gamma}\leq N} }\frac{\Upsilon_f[u](\psi)(y)}{S(u)}(\Lambda^{\mbX}_{s,t})^{\star_N k}(u).
\end{split}
\end{equation}
By simply iterating the property \ref{Lie_property} over $\star$ powers of primitive elements and matching   the symmetry factor of $\langle z^{\beta_1}\star  \cdots \star z^{\beta_k}, u\rangle $. Plugging \eqref{eq_partial_expansion} inside \eqref{eq_total_expansion} one has
\begin{align*}
&\psi\circ \mu^{\mbX}_{s,t}(y)= \sum_{\substack{u\in \mathfrak{F}(\mathfrak{M})\\0\leq \vert u\vert_{\gamma}\leq N} }\frac{\Upsilon_f[u](\psi)(y)}{S(u)}(\exp^{\star}_{N}(\Lambda^{\mbX}_{s,t})) (u)\\&+\sum_{k=2}^{N+1}\int_{\Delta^{k}}\sum_{ \beta_1\,, \ldots ,\beta_{k}\in A_k} \left(\frac{\Upsilon_f[z^{\beta_1}\star\cdots\star z^{\beta_{k}} ]}{S(z^{\beta_1})\cdots S(z^{\beta_{k}})} \right)(\psi)(Z_{r_{k}})\prod_{l=1}^{k} \Lambda^{\mbX}_{s,t}(z^{\beta_l}) \mathrm{d} r_{k}\ldots \mathrm{d} r_1\,\\&=\sum_{\substack{u\in \mathfrak{F}(\mathfrak{M})\\0\leq \vert u\vert_{\gamma}\leq N} }\frac{\Upsilon_f[u](\psi)(y)}{S(u)} \mbX_{s,t}(u) + (I)
\end{align*}
Then the theorem follows by simply remarking   that the term $(I)$ can be easily estimated  using the continuity  of $\psi$ and $f$ together with their derivatives and using the hypothesis $ K$  and $Z([0,1])$ are compact sets. In addition we have as a consequence of Proposition \ref{prop_estimate} the estimate
\[|\Lambda^{\mbX}_{s,t}( z^{\beta_l})|\leq \Vert\Lambda^{\mbX}\Vert^{|\beta_l|}|t-s|^{|\beta_l|_{\gamma}}\]
 from which we deduce \eqref{taylor_exp_flow} using the constraints $|\beta_1 |_{\gamma}+\cdots +|\beta_k |_{\gamma} > N$ given by the sets $A_k$. Passing to the other estimate \eqref{taylor_exp_flow_2}, we can repeat word by word the same identities  and the extra term $|y-w|$ arises because the term $(I) $ involve functions that are local Lipschitz.   The case  $f\in \CC^{\alpha_1}_b(\mathbb{R}, \mathbb{R}^d) $, $f_0 \in \mathcal{C}^{\alpha_2}_b$ or linear and $\psi\in \CC^{\alpha}_b $ follows easily by checking that all the estimate depend on the  norms $\|f\|_{\alpha_1}$ $\|f_0\|_{\alpha_2} $ and $\|\psi\|_{\alpha_3}$.
\end{proof}

\begin{remark}
We note that in this proof we find the smallest  regularity of the vector fields $f_0$ and $f$ to obtain the identity \eqref{eq_total_expansion}. This problem was of course not present in \cite[Lemma 4.7]{KL23} which was formulated in a manifold setting.
\end{remark}
Even if the  log-ODE almost flow \eqref{def_logode} does not have in general an explicit expression, using standard  results  on ordinary differential equations and  the previous property, we can actually establish some key-properties.

\begin{proposition} \label{properties_lamost_flow}
For any multi-index rough path $\mbX$ of  regularity $\gamma\in (0,1)$ and   $f\in \CC^{\alpha_1}(\mathbb{R}, \mathbb{R}^d) $, $f_0 \in \mathcal{C}^{\alpha_2}$  vector fields with  $\alpha_1>N_{\gamma}$ and $ \alpha_2>1$,  the log-ODE almost flow $\mu^{\mbX}$ satisfies following properties:
\begin{itemize}
\item (local Lipschitz continuity) for all compact sets $K\subset \mathbb{R}$  and $\abs{t-s}\le T^{\mbX}(K)$ there exists a  constant $L(K,s,t)>0$ with $L(K,s,s) = 1$ such that 
   \begin{equation*}
            \abs{\mu^{\mbX}_{s,t} (x)-\mu^{\mbX}_{s,t} (y)}\le L(K,s,t)\abs{x-y}\,
        \end{equation*}
for all $x,y\in K$.
\item (H\"older continuity) For all compact sets $K\subset \mathbb{R}$ there exists a constant $B(K)>0$  for which one has
        \begin{equation*}
            \abs{\mu^{\mbX}_{s,t} (x)-x} \le B(K)\abs{t-s}^\gamma\,
        \end{equation*}
for all $x\in K$.
\item (almost-flow property) For all compact set $K\subset \mathbb{R}$ there exists a constant $C(K)>0$ for which one has 
\begin{equation*}
            \abs{\mu^{\mbX}_{s,t}(x)-\mu^{\mbX}_{r,t}\circ\mu^{\mbX}_{s,r} (x)}\le C(K)\abs{t-s}^{(N_{\gamma} +1)\gamma}\,.
 \end{equation*} 
         \begin{equation*}
            \abs{(\mu^{\mbX}_{s,t}-\mu^{\mbX}_{r,t}\circ\mu^{\mbX}_{s,r})(x)-(\mu^{\mbX}_{s,t}-\mu^{\mbX}_{r,t}\circ\mu^{\mbX}_{s,r})(y)} \le C(K) \abs{t-s}^{(N_{\gamma} +1)\gamma}\abs{x-y}
        \end{equation*}
        for all $x, y\in K$, $\abs{t-s}\le T^{\mbX}(K)$ and all   $((r,t),\mu^{\mbX}_{r,s}(x))$, $((r,t),\mu^{\mbX}_{r,s} (y))\in O^{\mbX}$.
\end{itemize}
Moreover, in case  $f\in \CC^{\alpha_1}_b(\mathbb{R}, \mathbb{R}^d) $, $f_0 \in \mathcal{C}^{\alpha_2}_b$   or linear all these properties hold independently on the compact set $K$.
\end{proposition}

\begin{proof}
The local Lipschitz continuity follows from  the  local Lipschitz continuity of the log-ODE almost flow. Concerning the local H\"older continuity this follows immediately from the identity \eqref{first_expansion} with $\psi=\id$ and doing a simple estimation over powers of $|t-s|$. Concerning the almost-flow property, it is a simple consequence of Proposition \ref{main_prop} and the identity \eqref{eq:chen}. Indeed for all $x\in K$, $\abs{t-s}\le T^{\mbX}(K)$ and  $((r,t),\mu^{\mbX}_{r,s}(x))\in O^{\mbX}$ using the shorthand notation $N=N_{\gamma}$ we simply add and subtract the right quantities to obtain 

\begin{align}\label{eq_first_expansion}\nonumber
&\mu^{\mbX}_{r,t}\circ\mu^{\mbX}_{s,r} (x)=\\& \sum_{\substack{ z^{\beta}\in \mathfrak{M}\\0\leq \vert z^{\beta}\vert_{\gamma}\leq N}}\sum_{\substack{u\in \mathfrak{F}(\mathfrak{M})\\0\leq \vert u\vert_{\gamma}\leq N- |\beta|_{\gamma}}}\frac{\Upsilon_f[u](\Upsilon_f[z^{\beta}](\id))(x)}{S(u)S(z^{\beta})} \mbX_{s,r}(u)  \mbX_{r,t}( z^{\beta})  + (I)+ (II)
\end{align}
with the remainder terms
\[(I)= \mu^{\mbX}_{r,t}(\mu^{\mbX}_{s,r} (x))- \sum_{\substack{ z^{\beta}\in \mathfrak{M}\\0\leq \vert z^{\beta}\vert_{\gamma}\leq N}}\frac{\Upsilon_f[z^{\beta}](\id)(\mu^{\mbX}_{s,r} (x))}{S(z^{\beta})}\mbX_{r,t}( z^{\beta}) \]
\begin{align*}
(II)&=\sum_{\substack{ z^{\beta}\in \mathfrak{M}\\0\leq \vert z^{\beta}\vert_{\gamma}\leq N}}\frac{\Upsilon_f[z^{\beta}](\id)(\mu^{\mbX}_{s,r} (x))}{S(z^{\beta})}\mbX_{r,t}( z^{\beta})\\&- \sum_{\substack{ z^{\beta}\in \mathfrak{M}\\0\leq \vert z^{\beta}\vert_{\gamma}\leq N}}\sum_{\substack{u\in \mathfrak{F}(\mathfrak{M})\\0\leq \vert u\vert_{\gamma}\leq N- |\beta|_{\gamma}}}\frac{\Upsilon_f[u](\Upsilon_f[z^{\beta}](\id))(x)}{S(u)S(z^{\beta})} \mbX_{s,r}(u)  \mbX_{r,t}( z^{\beta})\,.
\end{align*}
Before focusing on the remainder terms let us show how to conclude from \eqref{eq_first_expansion}. By simply applying the composition property \eqref{morphism_element} we can rewrite the last term as 
\begin{equs}
& \sum_{\substack{ z^{\beta}\in \mathfrak{M}, u\in \mathfrak{F}(\mathfrak{M})\\0\leq |\beta|_{\gamma}+ |u|_{\gamma}\leq N}}\frac{\Upsilon_f[u\star z^{\beta}](\id)(x)}{S(u)S(z^{\beta})} \mbX_{s,r}(u) \mbX_{r,t}( z^{\beta}) \\& =\sum_{\substack{ z^{\beta}\in \mathfrak{M}\\0\leq \vert z^{\beta}\vert_{\gamma}\leq N}}\frac{\Upsilon_f[z^{\beta}](\id)(x)}{S(z^{\beta})} (\mbX_{s,r}\star_N\mbX_{r,t})( z^{\beta}) \\&=\mu^{\mbX}_{s,t}(x)+ \sum_{\substack{ z^{\beta}\in \mathfrak{M}\\0\leq \vert z^{\beta}\vert_{\gamma}\leq N}}\frac{\Upsilon_f[z^{\beta}](\id)(x)}{S(z^{\beta})} \mbX_{s,t}( z^{\beta})- \mu^{\mbX}_{s,t}(x)=\mu^{\mbX}_{s,t}(x)+ (III)
\end{equs}
where the third line follows from the simple fact that $\Upsilon_f[u](\id)=0$ for any $u\in \mathfrak{F}(\mathfrak{M})$ such that $u=u_1\bullet u_2
$. Then the remainder $(III)$ satisfies $|(III)|\leq C(K)\abs{t-s}^{(N_{\gamma} +1)\gamma}$ as consequence of \eqref{taylor_exp_flow}. Passing to the remainders $(I) $ and $(II)$, the first one satisfies $|(I)|\leq C(K)\abs{t-s}^{(N_{\gamma} +1)\gamma}$ directly  from \eqref{taylor_exp_flow}. In case of remainder $(II)$ we can  check from Proposition \ref{main_prop} that we can easily prove for any $z^{\beta}\in \mathfrak{M}$ the estimate
\begin{equs} |\Upsilon_f[z^{\beta}](\id)(\mu^{\mbX}_{s,r} (x))  & -\sum_{\substack{u\in \mathfrak{F}(\mathfrak{M})\\0\leq \vert u\vert_{\gamma}\leq N- |\beta|_{\gamma}}}\frac{\Upsilon_f[u](\Upsilon_f[z^{\beta}](\id))(x)}{S(u)} \mbX_{s,r}(u) |\\ & \leq C(K)|u-s|^{N- |\beta|_{\gamma} +1}
\end{equs}
for some $C(K)>0$. Therefore we can conclude there exists a constant $C'(K)>0$ such that
\begin{align*}
&|(II)|\leq \sum_{\substack{ z^{\beta}\in \mathfrak{M}\\0\leq \vert z^{\beta}\vert_{\gamma}\leq N}}C(K)|r-s|^{N- |\beta|_{\gamma} +1}\left|\frac{\mbX_{r,t}( z^{\beta})}{S(z^{\beta})}\right|\\&\leq \sum_{\substack{ z^{\beta}\in \mathfrak{M}\\0\leq \vert z^{\beta}\vert_{\gamma}\leq N}}C(K)|r-s|^{N- |\beta|_{\gamma} +1}\Vert\mbX\Vert^{|\beta|}|t-r|^{|\beta|_{\gamma}}\leq C'(K) |t-s|^{N +1}\,.
\end{align*}
To obtain the second part of the almost-flow property, we can easily repeat the same computations and apply the second estimate \eqref{taylor_exp_flow_2} instead. We finally remark that when   $f\in \CC^{\alpha_1}_b(\mathbb{R}, \mathbb{R}^d) $, $f_0 \in \mathcal{C}^{\alpha_2}_b$   the uniformity over the compact set holds as a consequence of Proposition \ref{main_prop}.
\end{proof}

Using the previous properties of the log-ODE almost flow we can then apply the sewing lemma for flows, see \cite{KL23,B15} to construct a unique local flow solution.

\begin{theorem}\label{thm_well-posedness}
For any multi-index rough path $\mbX$ of  regularity $\gamma\in (0,1)$ and  $f\in \CC^{\alpha_1}(\mathbb{R}, \mathbb{R}^d) $, $f_0 \in \mathcal{C}^{\alpha_2}$  vector fields with  $\alpha_1>N_{\gamma}$ and $ \alpha_2>1$ there exists a unique local flow solution of the rough differential equation \eqref{RDE}. Moreover in case  $f\in \CC^{\alpha_1}_b(\mathbb{R}, \mathbb{R}^d) $, $f_0 \in \mathcal{C}^{\alpha_2}_b$ or linear   there exists a unique global flow solution of the rough differential equation \eqref{RDE}.
\end{theorem}
\begin{proof}
The theorem follows by applying the sewing lemma for local almost-flows, see \cite[Lemma 2.7]{KL23} at the log-ODE almost flow. Indeed  the properties listed in Proposition \ref{properties_lamost_flow} are the hypothesis for the sewing lemma for local almost-flow, which imply the existence of an admissible set  $\hat{O}^{\mbX}\subset O^{\mbX}$ with explosion times $\hat{T}(K)$ and the existence and uniqueness of  a local real flow $\phi\colon \hat{O}^{\mbX}\to \mathbb{R}$ satisfying for any compact set $K\subset\mathbb{R}$ the following properties
  \begin{align}
            \sup_{x\in K}\abs{\phi_{s,t} (x)-\mu^{\mbX}_{s,t} (x)}&\le C(K)\,|t-s|^{(N_{\gamma} +1)\gamma}\,\label{local_flow_approx}\\
          \sup_{x\in K}\abs{\phi_{s,t} (x)-x}&\le B(K)\,|t-s|^{\gamma}\,\label{local_flow_approx_2}\\    \sup_{x,y\in K}\abs{\phi_{s,t} (x)-\phi_{s,t} (y)}&\le L(K)\,|x-y|\, \label{local_flow_approx_3}
        \end{align}
       for some constant $C(K),B(K), L(K)>0$. The resulting flow satisfies the property \eqref{DefnSolRDEGeneral} as a trivial consequence of 
       Proposition \ref{main_prop} combined with \eqref{local_flow_approx}. Moreover, one can see that supposing  the property \eqref{DefnSolRDEGeneral}  is 
       true for some local real flow, we can derive the conditions \eqref{local_flow_approx} and \eqref{local_flow_approx_2} which guarantee the uniqueness of a local flow via the uniqueness of the sewing lemma. When $f\in \CC^{\alpha_1}_b(\mathbb{R}, \mathbb{R}^d) $, $f_0 \in \mathcal{C}^{\alpha_2}_b$ or linear then the  properties of Proposition \ref{properties_lamost_flow} hold uniformly over the compact set and  we can apply to the log-ODE almost flow the sewing lemma for flows from \cite[Theorem 2.1]{B15} from which we obtain  the properties \eqref{local_flow_approx} and \eqref{local_flow_approx_2} and \eqref{local_flow_approx_3} uniformly in  $K$.  Hence, we have the existence and uniqueness of a global flow solution.
\end{proof}
\begin{remark}
Thanks to the local Lipschitz property \eqref{local_flow_approx_3} it possible to check that the  resulting flow  takes values in the space of locally  Lipschitz continuous homeomorphisms of $\mathbb{R}$ and depends continuously on the underlying rough path $\mathbf{X}$ from standard continuity result of the sewing lemma for flows, by using the  the metric $d(\mbX,\mathbf{Z})=\Vert\mbX-\mathbf{Z}\Vert $, with $\Vert \cdot \Vert$ from \eqref{defn_Norm}, see e.g. \cite[Theorem 3.8]{B15}.
\end{remark}

\subsection{Davie solution for a multi-index RDE}
Once the existence and uniqueness of a flow solution for the equation \eqref{RDE} has been obtained, we can easily relate it with a path solution solving \eqref{RDE} as a Davie solution, see \cite{Davie} and provide an equivalent formulation. This phenomenon holds true for rough differential equation driven by a both a geometric and branched rough path, see \cite{B21}. In this part  we will briefly explain how to adapt these results for multi-index rough paths.

To understand how to define a Davie solution we simply go back to the expansion \eqref{approximate_exp2} and we use it into its original version of Taylor expansion for any solution.

\begin{definition}[Davie solution] \label{Davie_solution}
 Let  $\mbX$ be a multi-index rough path of regularity $\gamma$ and $f\in \CC^{\alpha_1}(\mathbb{R}, \mathbb{R}^d) $, $f_0 \in \mathcal{C}^{\alpha_2}$  vector fields with  $\alpha_1>N_{\gamma}-1$ and $ \alpha_2>0$.  For any given initial datum  $y_0\in \mathbb{R}$ a path $Y\colon [0, \tau]\to \mathbb{R}$ is said a local Davie solution of  the rough differential equation \eqref{RDE} if $Y_0=y_0$ and there exists a constant $a>1$ independent of $\mbX$ and a positive constant  $c>0$ dependent possibly on $\mbX$ and $f$ such that
\begin{equation}
\label{DefnSolRDEGeneral_davie}
|Y_t- Y_s- \sum_{\substack{ z^{\beta}\in \mathfrak{M}\\1\leq \vert z^{\beta}\vert_{\gamma}\leq N_{\gamma}}}\frac{\Upsilon_f[z^{\beta}](Y_s)}{S(z^{\beta})} \mbX_{s,t}( z^{\beta})| \leq c\,|t-s|^a
\end{equation}
for all $ 0\leq s\leq t\leq \tau$. If  $\tau =T$  we say that  $z$ is a global Davie solution.

\end{definition}
From this definition we can easily relate the existence and uniqueness of Theorem \ref{thm_well-posedness} into this context.

\begin{theorem} \label{prop_Davie}
For any multi-index rough path $\mbX$ of  regularity $\gamma\in (0,1)$ and  $f\in \CC^{\alpha_1}(\mathbb{R}, \mathbb{R}^d) $, $f_0 \in \mathcal{C}^{\alpha_2}$  vector fields with  $\alpha_1>N_{\gamma}$ and $ \alpha_2>1$  there exists a unique local Davie solution  of the rough differential equation \eqref{RDE}. Moreover in case  $f\in \CC^{\alpha_1}_b(\mathbb{R}, \mathbb{R}^d) $, $f_0 \in \mathcal{C}^{\alpha_2}_b$ or linear there exists a unique global Davie solution of the rough differential equation \eqref{RDE}.

\end{theorem}

\begin{proof}
In order to prove this result we first show that for any given flow solution $\varphi$ one can define a Davie solution  $Y$. Given a flow solution $\varphi$ we can simply check that for any $y_0\in \mathbb{R}$ the path $Y_t=\varphi_{0,t}(y_0)$ defined for $t\leq T(\{y_0\})$ is indeed a path starting from $y_0$ for which for any $0\leq s\leq t\leq \tau$ one has 
\begin{align*}
Y_t=\varphi_{s,t}(\varphi_{0,s}(y_0))&=\sum_{\substack{ z^{\beta}\in \mathfrak{M}\\0\leq \vert z^{\beta}\vert_{\gamma}\leq N}}\frac{\Upsilon_f[z^{\beta}](\id)(\varphi_{0,s}(y_0))}{S(z^{\beta})} \langle\mbX_{s,t}, z^{\beta}\rangle+O(|t-s|^a)\\&= Y_s+\sum_{\substack{ z^{\beta}\in \mathfrak{M}\\1\leq \vert z^{\beta}\vert_{\gamma}\leq N}}\frac{\Upsilon_f[z^{\beta}](Y_s)}{S(z^{\beta})} \langle\mbX_{s,t}, z^{\beta}\rangle+O(|t-s|^a)\,.
\end{align*}
Conversely, supposing the existence of a Davie solution for any $y_0$ we can easily show that this Davie solution coincides with the local  solution  flow $\varphi_{0,t}(y_0)$ for $t\leq T(\{y_0\})\wedge\tau= \tau_1$, by using the same uniqueness argument of \cite[Theorem 5.2]{B15}. Using the local Lipschitz property of $\varphi$
for any value $\delta$ and any integer $k\leq \tau_1/\delta$ one has
\begin{equation*}
\begin{split}
Y_{k\delta} &= \varphi_{(k-1)\delta,\delta k}(Y_{(k-1)\delta}) + O\big(\delta^a\big) \\                                  &= \varphi_{(k-1)\delta,\delta k}\left(\varphi_{(k-2)\delta,(k-1)\delta}(Y_{(k-2)\delta})+ O\big(\delta^a\big) \right) + O\big(\delta^a\big) \\&=\varphi_{(k-2)\delta,\delta k}(Y_{(k-2)\delta})+ L(\{y_0\})O\big(\delta^a\big)   O\big(\delta^a\big)\,.
\end{split}
\end{equation*}
Iterating these estimate up to reach zero one has
\begin{equation*}
\begin{split}
Y_{k\delta} &= \varphi_{0,\delta k}(y_0) + L(\{y_0\})O\big(k\delta^a\big)  + O\big(\delta^a\big)\,.
\end{split}
\end{equation*}
By choosing now a sequence $\delta_n$ and $k_n$ such that $\delta_nk_n\to t$ with $t\leq\tau_1$ one has $Y_{t} = \varphi_{0,t}(y_0)$ by continuity.
\end{proof}
\section{Translation and action on RDE}\label{Sec::4}
The translation of rough paths is an algebraic deformation of rough paths. Translations in geometric and branched rough paths have been studied in \cite{BCFP} and the case of ``smooth rough paths" has been explored by \cite{Bel22}.
One of the main results in these works is that a rough differential equation driven by translated rough paths admits the same solution as the original equation with a corresponding translation in its nonlinearity $f$. 
This result can be applied to obtain a conversion between Itô and Stratonovich solutions and to the renomalisation of rough paths where some iterated integrals diverge (see \cite{BCF}).
The similar structure can be also found in \cite{BHZ} which introduces the renormalisation of singular SPDEs in regularity structures.

In this section, we discuss the translation of multi-indices rough paths and derive the same results as those of geometric and branched rough paths. In particular, we will analyse the effect of the translation on the solutions of RDEs.

\subsection{Translations and the simultaneous insertion product}
Suppose $\{e_0, \ldots e_d\}$ is  an orthonormal basis of $\mathbb{R}^{d+1}$ and 
$X_t $ is a $d+1$- dimensional path with and $X^0_t = t$, i.e.,
\begin{equs}
	X_t = \sum_{i=0}^{d} X^i_te_i = te_0 + \sum_{i=1}^{d} X^i_te_i.
\end{equs}
We define a translation $T_v: \mathbb{R}^{d+1} \rightarrow \mathbb{R}^{d+1} $ of the path $X_t$ as
\begin{equs}
	T_v X_t :=  X_t +v = \sum_{i=0}^{d} (X_t^i + \sum_{j=0}^{d} X_t^j v_j^i) e_i,  \quad \text{with } v_j = \sum_{i=0}^{d} v_j^ie_i \in \mathbb{R}^{d+1}.
\end{equs}
This can be considered as a perturbation of signals $X_t^i$.  Moreover, since
\begin{equs}
	X_t +v = \sum_{j=0}^{d} X_t^j (e_j + \sum_{i=0}^{d} v_j^ie_i),
\end{equs}
the translation is equivalent to endomorphisms of the vector space $\mathbb{R}^{d+1}$:
\begin{equs}
	e_j \mapsto e_j + v_j, \quad j=0,\ldots,d.
\end{equs}

Now, we would like to study how the translation in signals will affect multi-indices rough paths lifted from them.
Notice that we can identify the vector space of $\mathbb{R}^{d+1}$ with the space $\CH_{(1)}$ since $\{z_{(i,0)}\}_{i = 0, \ldots,d}$ is a natural basis of $\CH_{(1)}$. Therefore, one can define a translation $T_{\ell}: \CH_{(1)} \rightarrow \mathfrak{g}(\CH)$ where $\mathfrak{g}(\CH)$ is the set of primitive elements of $\CH$. 
We firstly define the translation of the basis as below and then take the linear extension. For any $z_{(i,0)} \in \CH_{(1)}$ and $i = 0, \ldots,d$, define the linear map 
	\begin{equs} \label{eq:signle_translation}
		T_{\ell}( z_{(i,0)} ) := \sum_{ z^{\alpha} 
			\in \mathfrak{M}} \frac{\ell_i(z^{\alpha})}{S(z^{\alpha})} z^\alpha,
	\end{equs}
where $ \ell_i : \mathfrak{g}(\CH) \rightarrow \mathbb{R}$ is a forest product-preserving (i.e., $\ell_i (\prod^\bullet_j z^{\beta_j}) = \prod_j \ell_i (z^{\beta_j})$) linear character  and it describes the translation in the direction $i$. To examine the translation effect on high-order terms of a rough path $\mathbf{X}$, 
one can extend \eqref{eq:signle_translation} to $T_{\ell}: \CH \rightarrow \CH$
by defining, for any $z^\beta \in \mathfrak{M}$,
\begin{equs} \label{general_tanslation_map}
	\begin{aligned}
	&\hat T_{\ell}( z_{(i,k)} ) = \sum_{ z^{\alpha} 
		\in \mathfrak{M}} \frac{\ell_i(z^{\alpha})}{S(z^{\alpha})} D^k z^\alpha,
	\quad
	T_{\ell}( z^\beta) = \prod_{(i,k) \in  [0,d] \times \mathbb{N}} \left(\hat T_{\ell}( z_{(i,k)} )\right)^{\beta(i,k)},
	\\&
	T_{\ell}( z^\alpha \bullet z^\beta) = T_{\ell}( z^\alpha) \bullet T_{\ell}( z^\beta)
	\end{aligned}
\end{equs}
 where the product $\prod_{(i,k) \in  [0,d] \times \mathbb{N}}$ is a multi-index product and $ \ell_i : \CH \rightarrow \mathbb{R}$ is a linear character preserving the forest product. One can check that, because of the $D^k$,  $T_{\ell}$ maps (linear span of) populated multi-indices to (linear span of) populated multi-indices. The map $T_{\ell}$
is also a morphism for $\triangleright$, i.e.,  for any $z^\alpha, z^\beta \in \mathfrak{M}$:
\begin{equs}\label{eq:morphism_insertion}
	\begin{aligned}
	T_{\ell}( z^\beta \triangleright z^\alpha) & = T_{\ell}( z^\beta) \triangleright T_{\ell}( z^\alpha)
	\end{aligned}
\end{equs}
In fact, since the populated multi-indices with $\triangleright$ of the present paper are the free Novikov algebra with $d+1$ generators, $T_{\ell}$ is the unique extension of \eqref{eq:signle_translation} as a Novikov morphism.

To solve RDEs driven by translated rough paths, having only the algebraic structures is not enough.  Measuring the regularity of translated rough paths is also necessary. Due to \cite[Thm 4.16]{Li23} and the fact that $|\cdot| \le |\cdot|_\gamma$, we have the following proposition which fulfils the role. 
\begin{proposition}\label{prop:regularity_translation}
	Let $\mbX$ be a multi-index rough path of H\"older regularity $\gamma$ and $N \in \mathbb{N}$ and $\{\ell_i\}_{i=0,\ldots, d}$ be a collection of characters $ \ell_i \colon \mathfrak{F} \rightarrow \mathbb{R} $.	  If $\ell_i(z^\alpha) = 0$ for any $|z^\alpha|_{\gamma} > N$ then    $T_{\ell}\mbX$ is a multi-index rough path of H\"older regularity $ \gamma/N$.
\end{proposition}

In the sequel, we focus on a simpler case where the translation $T_\ell^0$ happens only in the signal $t$. 
In other words, $\ell_i(z^\alpha) = \delta_{z^\alpha,z_{(i,0)}}$ for any $i \ne 0$. This case is where most applications (renormalisation, Itô-Stratonovich correction) come from. However, results in this section can be generalised to the general translation, provided all the algebraic objects are generalised in a suitable way.
Moreover, we would like to describe the translation operator defined in \eqref{general_tanslation_map} formally using some Hopf algebra. The first step is to focus on the ``generators" $\hat{T}^0_\ell$ and 
introduce the underlying pre-Lie product $ \blacktriangleright:  \mathfrak{M} \times \mathfrak{M} \rightarrow \mathrm{Span}_{\mathbb{R}}(\mathfrak{M})$ 
\begin{equs} \label{insertion_product}
	z^\alpha \blacktriangleright z^\beta : =\sum_{k \in  \mathbb{N}}\left(D^{k}z^{\alpha}\right)
	\left(\partial_{z_{(0,k)}}z^\beta\right)
\end{equs}
where $\partial_{z_{(0,k)}}$ is the partial derivative with respect to the variable $z_{(0,k)}$ (i.e.,  for any $z^\beta \in \mathfrak{M}$, if $\alpha(z_{(0,k)}) \ne 0$ then $ \partial_{z_{(0,k)}}z^\beta = \frac{z^{\beta}}{z_{(0,k)}} $, otherwise $ \partial_{z_{(0,k)}}z^\beta = 0$). 
Therefore, $\blacktriangleright$ represents the operation of replacing the variables $z_{(0,k)}$ by $D^{k}z^{\beta}$, which is the main idea of $\hat{T}_\ell^0$. This is also in accordance with the fact that, while renormalising the rough path, each divergent part is contracted to a single variable decorated by the $0$-th driven process $X^0 = t$ (see Section \ref{sec:example} below).  Notice that $z^\beta, z^\alpha \in \mathfrak{M}$ indicate that none of them is the empty forest. 
A similar product to \eqref{insertion_product} has been introduced in  \cite[Def. 4.3]{Li23} which is inspired by a product on trees in \cite[Sec. 3.4]{BM22}. Since both $z_{(i,k)}$ and $D^{k}$ have degree of $k$, and each of $z^{\alpha}$ and $z^{\beta}$ is populated, $z^\beta \blacktriangleright z^\alpha$ produces a linear combination of populated multi-indices.

Given a character $ \ell$, we can, therefore, describe the translation map $ T^0_{\ell} $ by defining $\hat{T}^0_{\ell}$ as the following.
\begin{equs}
	\hat{T}^0_{\ell}  \left(    z_{(i,k)} \right) = z_{(i,k)} \text{ for $i \neq 0$}, \quad 	\hat{T}^0_{\ell}  \left(    z_{(0,k)} \right) = 	\sum_{z^{\alpha} 
		\in \mathfrak{M}
	} \frac{\ell(z^{\alpha})}{S(z^{\alpha})}(z^\alpha \blacktriangleright z_{(0,k)}).
\end{equs}
From \cite[Lemma 4.2, 4.6]{Li23},
one proves \eqref{eq:morphism_insertion} and has
\begin{equs} \label{morphism}
	z^{\alpha} \blacktriangleright D^{k} z^{\beta} = D^k \left(	z^{\alpha} \blacktriangleright  z^{\beta} \right)
\end{equs}
which allows to say that
\begin{equs} \label{translation_map}
	T_{\ell}^0  \left(    z_{(i,0)} \right) = z_{(i,0)}, \quad 	
	T_{\ell}^0  \left(    z_{(0,0)} \right) =  \sum_{ z^{\alpha} 
		\in \mathfrak{M}} \frac{\ell(z^{\alpha})}{S(z^{\alpha})} z^\alpha 
\end{equs}
where we recover the definition given in \cite[Sec. 4.2]{Li23}. 
By the property \eqref{morphism}, one has for any $z^\alpha, z^\beta \in \mathfrak{M}$
\begin{equs} \label{morphism_induction}
	T_{\ell} (z^\alpha \star z^\beta) = (T_{\ell} z^\alpha) \star (T_{\ell} z^\beta)
\end{equs}
which, through induction on the norm of a multi-index, ensures the existence of  $T_{\ell} (z^\beta) $ generated by \eqref{general_tanslation_map} for any $z^\beta \in \mathfrak{M}$.
Then we have, for any $z^\beta \in \mathfrak{M}$,
\begin{equs} \label{dual_renormalisation}
	T^0_\ell (z^\beta) =  \sum_{  z^{\bullet\tilde \alpha} 
		\in \mathfrak{F}} \frac{\ell( z^{\bullet \tilde \alpha})}{S( z^{\bullet\tilde \alpha})}  z^{\bullet\tilde \alpha} \star_1 z^\beta,
\end{equs}
where $\star_1$ is the ``simultaneous insertion product" 	constructed upon the pre-Lie product $\blacktriangleright$  by the following definition.
\begin{definition} The simultaneous insertion product
	$\star_1: \mathfrak{F} \times \mathfrak{M} \rightarrow  \mathrm{Span}_{\mathbb{R}}(\mathfrak{M})$ is defined as
	\begin{equs} \label{multiple_insertion}
		\prod_{i =1}^{\bullet n} z^{\beta_i} \star_1 z^\alpha:= \sum_{k_1,...,k_n\in \mathbb{N}}  
		\left(\prod_{i=1}^nD^{k_i}z^{\beta_i}\right)\left[\left(\prod_{i=1}^n\partial_{z_{(0,k_i)}}\right)z^\alpha\right] 
	\end{equs}
	for $n \in \mathbb{N}_+$. 
	To describe the translation \eqref{general_tanslation_map} by \eqref{dual_renormalisation}, we also need the restriction that 
	\begin{equation}\label{convention2}
		n = |z^\alpha|_0 = |\alpha|_0,
	\end{equation}
	where $ |z^\alpha|_0 $ is the number of variables $ z_{(0,k)} $ in $ z^\alpha $.
	Moreover, to prevent the translation from killing the terms without any variable $z_{(0,k\in \mathbb{N})}$, we define the insertion of empty forest ($n=0$)
		\begin{equs} \label{eq:primative}
			\emptyset \star_1 z^{\alpha} := z^{\alpha}.
	\end{equs} 
\end{definition}
	The product $\star_1$ can be generalised to $\star_1: \mathfrak{F} \times \mathfrak{F} \rightarrow \mathrm{Span}_{\mathbb{R}}(\mathfrak{F})$, as we have the Leibniz rule on $D$ \eqref{Leibniz_D}.  Furthermore, we can take the linear extension to introduce $\star_1: \mathrm{Span}_{\mathbb{R}}(\mathfrak{F}) \times \mathrm{Span}_{\mathbb{R}}(\mathfrak{F}) \rightarrow \mathrm{Span}_{\mathbb{R}}(\mathfrak{F})$.
	\begin{remark}
		Recall that when we construct the associative product $\star$ from the pre-Lie product $\triangleright$, the  two-steps Guin-Oudom procedure described by \eqref{definition_star_2} and \eqref{definition_star} are used. Here the construction \eqref{multiple_insertion} is similar to the one of $\star_2$ by \eqref{definition_star_2}. Indeed, in both case, the elements $z^{\beta_i}$ in the forest multiply $z^\alpha$ independently, according to the corresponding pre-Lie product, and $z^{\beta_i}$ are forbidden from multiplying each other. Therefore, we can repeat the second step of the Guin-Oudom procedure \eqref{definition_star} to construct an associative product $\star'$ upon $\star_1$ by
		\begin{equs}
			\prod_{i=1}^{\bullet n} z^{\beta_i} \star'  z^{\bullet \tilde \alpha} := \mu  (\id  \otimes \cdot \star_1   z^{\bullet \tilde \alpha})  \Delta_\shuffle \left(\prod_{i=1}^{\bullet n}z^{\beta_i}\right).
		\end{equs}
\end{remark}

\subsection{Dual operators}	
	Recall that a multi-indices rough path  $\mathbf{X} \in \CH$ has the form
	\begin{equs}
		\mathbf{X} = \sum_{z^\beta \in \mathfrak{M}}\frac{{\mathbf{X}}(z^\beta)}{S(z^\beta)}z^\beta.
	\end{equs}
	One has $\langle \mathbf{X} , z^\beta \rangle = \mathbf{X}(z^\beta)$ and $\langle \mathbf{X} , z_{(i,0)} \rangle = X^i$. 
The translated rough path is given by 
	\begin{equs}
		T^0_\ell \mathbf{X} &= \sum_{z^\beta \in \mathfrak{M}}\frac{{\mathbf{X}}(z^\beta)}{S(z^\beta)}T^0_\ell z^\beta
	=
		\sum_{z^\beta \in \mathfrak{M}}
		\sum_{  z^{\bullet\tilde \alpha} \in \mathfrak{F}}
		\frac{{\mathbf{X}}(z^\beta)\ell( z^{\bullet \tilde \alpha})}{S(z^\beta)S( z^{\bullet\tilde \alpha})}
		z^{\bullet\tilde \alpha} \star_1 z^\beta
		.
	\end{equs}
	Now, if one wants to rewrite $T^0_\ell \mathbf{X}$ in the basis of $\mathfrak{M}$ as below
	\begin{equs}\label{eq:change_of_basis}
		T^0_\ell \mathbf{X} =  \sum_{z^\gamma \in \mathfrak{M}}\frac{\tilde {\mathbf{X}}(z^\gamma)}{S(z^\gamma)}z^\gamma,
	\end{equs}
	the correct coefficients $\tilde {\mathbf{X}}(z^\gamma) = \langle \tilde {\mathbf{X}}, z^\gamma \rangle$ should be determined. Therefore, it is necessary to identify those $z^{\bullet\tilde \alpha}$ and $z^\beta$ which produce $z^{\bullet\tilde \alpha} \star_1 z^\beta$ that contains a term of $z^\gamma$. This indicates the need for $\star_1$'s adjoint operator, called extraction-contraction coproduct $ \Delta^{\!-}: \mathrm{Span}_{\mathbb{R}}(\mathfrak{M}) \rightarrow \mathrm{Span}_{\mathbb{R}}(\mathfrak{F}) \otimes \mathrm{Span}_{\mathbb{R}}(\mathfrak{M})$ and is defined through the adjoint relation
\begin{equs}\label{eq:adjoint}
	\langle \CF \otimes z^\alpha, \Delta^{\!-} z^\beta \rangle = \langle \CF \star_1 z^\alpha,  z^\beta \rangle
\end{equs} 
for $z^\alpha, z^\beta \in \mathfrak{M}$ and $\CF \in \mathfrak{F}$. 
The explicit formula of $\Delta^{\!-}$ can be found in the same way as in \cite[Thm 3.10]{BH24}. We also require $\Delta^{\!-}$ to be linear in $\mathrm{Span}_{\mathbb{R}}(\mathfrak{M})$ since $\star_1$ is bilinear. Then we generalise the coproduct to the space $\mathfrak{F}$ by requiring $\Delta^{\!-}$ to preserve the forest product
\begin{equs} \label{eq:Delta_forest}
	\Delta^{\!-} (z^{\beta_1} \bullet z^{\beta_2}) := (\Delta^{\!-} z^{\beta_1})(\Delta^{\!-} z^{\beta_2})
\end{equs}
for any $z^{\beta_1}, z^{\beta_2} \in \mathfrak{M}$. 
Finally, one can generalise the extraction-contraction coproduct to $ \Delta^{\!-}: \mathrm{Span}_{\mathbb{R}}(\mathfrak{F}) \rightarrow \mathrm{Span}_{\mathbb{R}}(\mathfrak{F}) \otimes \mathrm{Span}_{\mathbb{R}}(\mathfrak{F})$ as the linear extension of the above-defined operator .This also generates automatically the adjoint relation to
\begin{equs}
	\langle \tilde \CF_1 \otimes \tilde \CF_2, \Delta^{\!-} \tilde \CF_3\rangle = \langle \tilde \CF_1 \star_1 \tilde \CF_2,  \tilde \CF_3 \rangle
\end{equs} 
for $\tilde \CF_1, \tilde \CF_2,  \tilde \CF_3 \in \mathrm{Span}_{\mathbb{R}}(\mathfrak{F})$ by using \eqref{eq:Delta_forest}.

Following the adjoint relation \eqref{eq:adjoint} between $\Delta^{\! -}$ and $\star_1$, 
we can find the adjoint operator of $T_{\ell}^0$ through
\begin{equs}
	\langle T^0_{\ell}z^\beta, z^\gamma \rangle =	\langle z^\beta, M^0_{\ell}z^\gamma \rangle
\end{equs}
where  $M_{\cdot}^0$ maps characters $\ell$ to a linear operator $M_{\ell}^0: \mathrm{Span}_{\mathbb{R}}(\mathfrak{F}) \rightarrow \mathrm{Span}_{\mathbb{R}}(\mathfrak{F})$
 with the following explicit expression.
\begin{equs}\label{eq:renormalisation}
	M_{\ell}^0  = \left(\ell \otimes \id \right)\Delta^{\!-}.
\end{equs}
One can also derive an adjoint relation for the general translation 
\begin{equs}\label{adjoint}
	\langle T_{\ell}z^\beta, z^\gamma \rangle =	\langle z^\beta, M_{\ell}z^\gamma \rangle
\end{equs}
with $	M_{\ell}  = \left(\ell \otimes \id \right) \hat\Delta^{\!-}.$
The explicit formula of $\hat\Delta^{\!-}$ can be obtained by the same method in \cite[Thm 3.10]{BH24}.
By \eqref{adjoint}, we have 
\begin{equs}
	T_{\ell}\mathbf{X}(z^\beta) =  \mathbf{X}(M_\ell z^\beta)
\end{equs}
which gives the coefficients in the basis-change problem \eqref{eq:change_of_basis}.

\subsection{Translations in rough differential equations}\label{sec:example}
Finally, we can demonstrate the equivalence between the translation in rough paths and the translation in the nonlinearity $f$ of a rough differential equation.
\begin{theorem}\label{thm:renormalization_RDE}
Given some integer $N \in \mathbb{N}_+$. Let  $\mbX$ be a multi-index rough path of H\"older regularity $\gamma$ and $f\in \CC^{\alpha_1}(\mathbb{R}, \mathbb{R}^d) $, $f_0 \in \mathcal{C}^{\alpha_2}$  vector fields with  $\alpha_1>NN_{\gamma}$, $ \alpha_2>1$. For any given $\{\ell_i\}_{i=0,\ldots, d}$ which is a collection of characters $ \ell_i \colon \mathfrak{F} \rightarrow \mathbb{R} $ and $i = 0,\ldots,d$  such that $\ell_i(z^\alpha) = 0$ for any $|z^\alpha|_{\gamma} > N$, a path $Y\colon[0,T]\to\mathbb{R} $ solves the equation
\begin{equs} \label{equa_renor}
dY_t = f(Y_t) d (T_{\ell}(\mbX_t))\,
\end{equs}
if and only if it solves 
$$
dY_t = f^{\ell} (Y_t) d \mbX_t, 
$$
where 
$$
f_i^{\ell} = \sum_{z^{\alpha} \in \mathfrak{M}} \frac{\ell_i(z^{\alpha})}{S(z^{\alpha})} \Upsilon_f[z^{\alpha}]\,.
$$

\end{theorem}
\begin{proof} Thanks to theorem \ref{prop_Davie} we can equivalently describe a local solution of \eqref{RDE} in terms of  Definition \ref{Davie_solution}. By simply applying this definition any local solution  $Y\colon [0, \tau]\to \mathbb{R}$ of \eqref{equa_renor} satisfies
	\begin{equation}\label{davie_sol_1}
		|Y_t- Y_s- \sum_{\substack{ z^{\beta}\in \mathfrak{M}\\1\leq \vert z^{\beta}\vert_{\gamma}\leq \lfloor N/ \gamma \rfloor}}\frac{\Upsilon_f[z^{\beta}](Y_s)}{S(z^{\beta})} (T_{\ell}\mbX_{s,t})( z^{\beta})| \leq c\,|t-s|^a
	\end{equation}
	for all $ 0\leq s\leq t\leq \tau$. 
	
Then, the main idea of the proof is performing change of basis to $T_{\ell}\mbX_{s,t}$ as indicated in \eqref{eq:change_of_basis}.
Because of the adjoint \eqref{adjoint} we can indeed rewrite the sum in the right-hand side of \eqref{davie_sol_1} as
	\begin{equs}
		 \sum_{\substack{z^{\beta}\in \mathfrak{M}\\1\leq \vert z^{\beta}\vert_{\gamma}\leq \lfloor N/ \gamma \rfloor}}\frac{\Upsilon_f[z^{\beta}](Y_s)}{S(z^{\beta})}  (T_{\ell} \mbX_{s,t}) (z^{\beta})
		 &=
		 \sum_{\substack{z^{\beta}\in \mathfrak{M} \\1\leq \vert z^{\beta}\vert_{\gamma}\leq \lfloor N/ \gamma \rfloor}}\frac{\Upsilon_f[z^{\beta}](Y_s)}{S(z^{\beta})}  \mbX_{s,t}(M_{\ell}z^{\beta})
		\\& =
		\sum_{\substack{z^{\beta}\in \mathfrak{M}\\1\leq \vert z^{\beta}\vert_{\gamma}\leq \lfloor 1/\gamma\rfloor}}\frac{\Upsilon_f[T_{\ell} z^{\beta}](Y_s)}{S(z^{\beta})} \mbX_{s,t}(z^{\beta} )\,.
	\end{equs} 
Finally we have to show
\begin{equs}
	\Upsilon_f[T_{\ell} z^{\beta}] = 	\Upsilon_{f^{\ell}}[ z^{\beta}].
\end{equs}
This follows from the definition of $T_{\ell}$ given in \eqref{general_tanslation_map} and by multiplicativity one has to check only for $ z_{(i,k)} $
\begin{equs}
	\Upsilon_f[T_{\ell} z_{(i,k)}] & =   \sum_{ z^{\alpha} 
		\in \mathfrak{M}} \frac{\ell_i(z^{\alpha})}{S(z^{\alpha})} \Upsilon_f[D^k z^\alpha]
	\\  & = \partial^k  \sum_{ z^{\alpha} 
		\in \mathfrak{M}} \frac{\ell_i(z^{\alpha})}{S(z^{\alpha})}  \Upsilon_f[ z^\alpha]
	\\ & = \partial^k   \Upsilon_{f^{\ell}}[z_{(i,0)}]
	\\ & = \partial^k   \Upsilon_{f^{\ell}}[ D^k z_{(i,0)}].
\end{equs}
Therefore we can rewrite \eqref{davie_sol_1} as
\begin{equation}\label{davie_sol_2}
		|Y_t- Y_s- \sum_{\substack{ z^{\beta}\in \mathfrak{M}\\1\leq \vert z^{\beta}\vert_{\gamma}\leq N_{\gamma}}}\frac{\Upsilon_{f^{\ell}}[z^{\beta}](Y_s)}{S(z^{\beta})} (\mbX_{s,t})( z^{\beta})| \leq c\,|t-s|^a
	\end{equation}
which is exactly the notion of Davie solution of \eqref{RDE} driven by $f^{\ell}$.
	\end{proof}
	 Theorem \ref{thm:renormalization_RDE} indicates that a path $Y\colon[0,T]\to\mathbb{R} $ solves the equation
		\begin{equs} 
			dY_t = f(Y_t) d \mathbf{B}^{\mathrm{Strat}} _t\,
		\end{equs}
		if and only if it solves respectively
		$$
		dY_t = f^{\ell} (Y_t) d  \mathbf{B}^{\mathrm{It\hat{o}}}_t, 
		$$
		where $\ell$ is given in \eqref{eq:correction_character} and
		$$
		f_i^{\ell} = \sum_{z^{\alpha} \in \mathfrak{M}} \frac{\ell_i(z^{\alpha})}{S(z^{\alpha})} \Upsilon_f[z^{\alpha}]\,.
		$$

To conclude, we discuss two applications of translations of multi-indices rough paths: Itô-Stratonovich correction and renormalisation. One example of Itô-Stratonovich correction with explicit computations will be given.
Itô integration and Stratonovich integration are two widely used interpretations to rough integrals. Itô integration produces martingales while Stratonovich integration obeys the integration by parts. The Stratonovich integral can be converted to an Itô integral by adding a ``Itô-Stratonovich correction" term.
\begin{equs}\label{eq:I_S_correction}
	\int_{s}^{t}\int_{s}^{r} \circ \mathrm{d}Z_u \circ \mathrm{d}X_r = \int_{s}^{t}\int_{s}^{r} \mathrm{d}Z_u  \mathrm{d}X_r +\frac{1}{2}\left[X,Y\right]_{s,t}
\end{equs}
where $\circ$ indicates Stratonovich integral and the $\left[X,Y\right]_{s,t}$ is the quadratic covariation of $X$ and $Y$ between time $s$ and $t$.  While solving a rough differential equation, one has to choose an interpretation of integrals according to the context of the equation. However, the Itô-Stratonovich correction can be described as a translation when one lifts the controls to rough paths. Theorem \ref{thm:renormalization_RDE} provides a explicit conversion between Itô and Stratonovich solutions of one RDE.

Let us consider the multi-indices rough paths  lifted from standard Brownian motion controls $\mathbf{B}$. We denote $\mathbf{B}^{\mathrm{Strat}}$ as the rough path lifted in the sense of Stratonovich integration and  $\mathbf{B}^{\mathrm{It\hat{o}}}$ is lifted in the sense of Itô integration. By \cite[Lemma 3.1]{Li23}, the unique lift of $\mathbf{B}^{\mathrm{Strat}}$  and the unique lift of $\mathbf{B}^{\mathrm{It\hat{o}}}$ are
\begin{equs} \label{eq:lift}
	&\mathbf{B}^{\mathrm{Strat}}_{s,t} (z^\beta)= \sum_{(i,k)} \sum_{\beta = e(i,k)+\beta_1+\ldots+\beta_k} \int_s^t \mathbf{B}^{\mathrm{Strat}}_{s,u} (z^{\beta_1})\ldots \mathbf{B}^{\mathrm{Strat}}_{s,u} (z^{\beta_k})\circ dB^i_u,
	\\&
	\mathbf{B}^{\mathrm{It\hat{o}}} = \sum_{(i,k)} \sum_{\beta = e(i,k)+\beta_1+\ldots+\beta_k} \int_s^t \mathbf{B}^{\mathrm{It\hat{o}}}_{s,u} (z^{\beta_1})\ldots \mathbf{B}^{\mathrm{It\hat{o}}}_{s,u} (z^{\beta_k})dB^i_u
\end{equs}
where $z^{\beta_j}$ are populated multi-indices, and $e(i,k)$ is the multi-index with the only non-zero entry to be $(i,k)$ and the frequency is $1$.

From \eqref{eq:I_S_correction} one can see that the Itô-Stratonovich correction corresponds to the following translation in rough paths.
\begin{align} 
&	\mathbf{B}^{\mathrm{Strat}} = T^0_{\ell} \mathbf{B}^{\mathrm{It\hat{o}}},
\\& \ell(z^\beta) = 
\begin{cases} \label{eq:correction_character}
 \frac{1}{2}\delta_{i,j}, & z^\beta = z_{(i,0)}z_{(j,1)} \text{ for } i,j=1,\ldots,d \\
	0, &\text{otherwise}
\end{cases}
\quad
\text{for } z^\beta \in \mathfrak{M}
\\&  \ell(\emptyset) = 1.
\end{align}
Then we can extend the valuation \eqref{eq:correction_character} to $\ell(\CF)$ with $\CF \in \mathfrak{F}$ since $\ell$ preserves the forest product. Then, according to the adjoint relation \eqref{adjoint}, we have 
\begin{equs} \label{correction_dual}
	\mathbf{B}^{\mathrm{Strat}} (z^\beta) = \mathbf{B}^{\mathrm{It\hat{o}}}(M_\ell^0 z^\beta). 
\end{equs}
 By \cite[Thm 3.10]{BH24}, we can calculate the explicit formula of $\Delta^{\!-}$. The only difference is that the variables produced from the contraction are specifically $z_{(0,k_i)}$ here. 
 To illustrate,  the coproduct acting on $z^\beta \in \mathfrak{M}$ is in the form of
 \begin{equs}\label{eq:coproduct}
 	\Delta^{\! -}(z^{\beta}) = \sum_{\substack{\CF \in \mathfrak{F}\\z^\alpha \in \mathfrak{M}}} n(\CF, z^\alpha, z^{\beta}) \CF \otimes z^\alpha\,,
 \end{equs}
 	where $n(\CF, z^\alpha, z^{\beta})$ is  a combinatorial factor counting the number of distinct extraction-contraction $z^{\beta}$ yielding $\CF \otimes z^\alpha$. The extraction-contraction is conducted as the following: 
 		\begin{enumerate}
 			\item Split the multi-index $\beta= \sum_{i=1}^{n} \hat{\gamma_i}$ for $ 1 \le n \le |z^\beta|$.
 			\item Extract each ``sub-multi-index" $z^{\hat{\gamma_i}}$ to a populated multi-index $z^{\gamma_i}$ which is obtained by applying $D^{k_i}$ to $z^{\hat{\gamma_i}}$. Put all the populated multi-indices together as a forest $\CF = \prod^{\bullet}_i z^{\gamma_i}$. Here all the sub-multi-indices are extracted because of the restriction \eqref{convention2} from its dual product $\star_1$.
 			\item contract each $z^{\hat{\gamma_i}}$ to a single variable $z_{(0,k_i)}$ and get $z^\alpha = \prod_{i=1}^nz_{(0,k_i)}$.
 			\item Add one extra term $\emptyset \otimes z^\beta$ due to \eqref{eq:primative}.
 		\end{enumerate}
Then, one can check the equality \eqref{correction_dual} explicitly through the formula \eqref{eq:coproduct}. In particular, \eqref{eq:I_S_correction} is recovered when $z^\beta = z_{(i,0)} z_{(j,1)}$.  We check the dual operator 
 	\begin{equation*}
 		M_\ell^0 z_{(i,0)} z_{(j,1)} = \left(\ell \otimes \id \right)\Delta^{\!-}z_{(i,0)} z_{(j,1)}  = \frac{1}{2}\delta_{i,j}z_{(0,0)} + z_{(i,0)} z_{(j,1)}. 
 \end{equation*}
 Then, according to the formulae \eqref{eq:lift}, the rough paths are lifted to
 \begin{equs}
 	  \mathbf{B}^{\mathrm{It\hat{o}}}(M_\ell^0 z_{(i,0)} z_{(j,1)}) &= \frac{1}{2}\delta_{i,j} (t-s) + \int_{s}^{t}\int_{s}^{r}  \mathrm{d}B^i_u \mathrm{d}B^j_r 
 	  \\&= 
 	  \int_{s}^{t}\int_{s}^{r} \circ \mathrm{d}B^i_u \circ \mathrm{d}B^j_r
 	  =
 	\mathbf{B}^{\mathrm{Strat}} (z_{(i,0)} z_{(j,1)}) 
 \end{equs}
As a second example, we consider the multi-indice $ z_{(i,0)}z_{(j,1)}z_{(k,1)} $. One has
 	\begin{equs}
 		\left(\ell \otimes \id \right)\Delta^{\!-} z_{(i,0)}z_{(j,1)}z_{(k,1)} &= z_{(i,0)}z_{(j,1)}z_{(k,1)}+ \frac{1}{2}\delta_{i,j}z_{(0,0)} z_{(k,1)}
 		\\&
 		+ \frac{1}{2}\delta_{i,k}z_{(0,0)} z_{(j,1)} + \delta_{j,k}z_{(i,0)} z_{(0,1)}
 	\end{equs}
 	Then, 
 	\begin{equs}
 		&\mathbf{B}^{\mathrm{It\hat{o}}}(M_\ell^0 z_{(i,0)}z_{(j,1)}z_{(k,1)} ) 
 		\\=& \int_{s}^{t}\int_{s}^{r}\int_{s}^{u}  \mathrm{d}B^i_v \mathrm{d}B^j_u \mathrm{d}B^k_r 
 		+ \int_{s}^{t}\int_{s}^{r}\int_{s}^{u}  \mathrm{d}B^i_v \mathrm{d}B^k_u \mathrm{d}B^j_r 
 		\\+&
 		 \frac{1}{2}\delta_{i,j} \int_{s}^{t}(r-s)\mathrm{d}B^k_r 
 		 + \frac{1}{2}\delta_{i,k} \int_{s}^{t}(r-s)\mathrm{d}B^j_r 
 		 +\delta_{j,k} \int_{s}^{t}\int_{s}^{r} B^i_u \mathrm{d}r
 	\end{equs}
 	On the other hand, 
 	\begin{equs}
 		&\mathbf{B}^{\mathrm{Strat}} (z_{(i,0)}z_{(j,1)}z_{(k,1)} ) 
 		\\=& \int_{s}^{t}\int_{s}^{r}\int_{s}^{u} \circ \mathrm{d}B^i_v \circ\mathrm{d}B^j_u \circ \mathrm{d}B^k_r 
 		+ \int_{s}^{t}\int_{s}^{r}\int_{s}^{u} \circ \mathrm{d}B^i_v \circ\mathrm{d}B^k_u \circ \mathrm{d}B^j_r 
 		\\= &
 		\int_{s}^{t}\left(\int_{s}^{r}\int_{s}^{u}  \mathrm{d}B^i_v \mathrm{d}B^j_u + \frac{1}{2}\delta_{i,j}(r-s)\right) \circ \mathrm{d}B^k_r 
 		\\+&
 		\int_{s}^{t}\left(\int_{s}^{r}\int_{s}^{u}  \mathrm{d}B^i_v \mathrm{d}B^k_u+\frac{1}{2}\delta_{i,k}(r-s)\right) \circ \mathrm{d}B^j_r 
 		\\= &
 		\frac{1}{2}\delta_{i,j} \int_{s}^{t}(r-s)  \mathrm{d} B^k_r 
 		+ \frac{1}{2}\delta_{i,k} \int_{s}^{t}(r-s)\mathrm{d}B^j_r 
 	    \\+&
 	    \int_{s}^{t}\int_{s}^{r}\int_{s}^{u}  \mathrm{d}B^i_v \mathrm{d}B^j_u \mathrm{d}B^k_r 
 	    + \int_{s}^{t}\int_{s}^{r}\int_{s}^{u}  \mathrm{d}B^i_v \mathrm{d}B^k_u \mathrm{d}B^j_r +\delta_{j,k} \int_{s}^{t}\int_{s}^{u} dB^i_v \mathrm{d}u
 	\end{equs}
 	where we applied the following the quadratic covariation 
 	\begin{equs}
 		\left[\int_{s}^{\cdot}\int_{s}^{u}  \mathrm{d}B^i_v \mathrm{d}B^j_u, B^k\right]_{s,t} = \frac{1}{2}\delta_{j,k}\int_{s}^{t}\int_{s}^{u} d B^i_v \mathrm{d}u = \left[\int_{s}^{\cdot}\int_{s}^{u}  \mathrm{d}B^i_v \mathrm{d}B^k_u, B^j\right]_{s,t}.
 	\end{equs}

Another example is renormalisation as stressed in \cite{BCF,BCFP} where the area of some rough paths needs to be renormalised.
The main idea there is to extract the divergent parts in the iterated integrals. This procedure can  be described with the same operator $M_{\ell}^0= \left(\ell \otimes \id \right)\Delta^{\!-}$
with a suitable character $\ell$. 
The formulation is  similar to the one given for renormalisation in regularity structures for decorated trees in \cite{BHZ}.

\end{document}